\documentclass[leqno]{amsart}
\usepackage{amsmath,amsfonts,amsthm,amssymb,indentfirst,epic,url,graphics,needspace}

\newtheorem{theorem}{Theorem}
\newtheorem{lemma}[theorem]{Lemma}
\newtheorem{corollary}[theorem]{Corollary}
\newtheorem{proposition}[theorem]{Proposition}
\newtheorem{definition}[theorem]{Definition}

\newtheorem{convention}[theorem]{Convention}

\newcommand{\dotline}[2]{\put(#1){\circle*{2}}\qbezier(#1)(#1)(#2)}

\newcommand{\Z}{\mathbb{Z}}
\newcommand{\N}{\mathbb{N}}
\renewcommand{\r}{\mathrm}
\newcommand{\pq}{\preccurlyeq}
\newcommand{\ad}[1]{\r{ad}_{#1}}
\newcommand{\B}[1]{\makebox[2em]{$#1$}}
\newcommand{\BB}[1]{\makebox[20em]{#1}}

\begin{document}

\begin{center}
\texttt{Comments, corrections,
and related references welcomed, as always!}\\[.5em]
{\TeX}ed \today
\vspace{2em}
\end{center}

\title%
{Some results on counting linearizations of posets}%
\thanks{
This note is readable online at
\url{http://math.berkeley.edu/~gbergman/papers/unpub}
and \url{http://arxiv.org/abs/arXiv:1802.01712}\,.
}


\author{George M. Bergman}
\address{University of California\\
Berkeley, CA 94720-3840, USA}
\email{gbergman@math.berkeley.edu}

\begin{abstract}
In section~\ref{S.ord+} we consider a $\!3\!$-tuple $S=(|S|,\pq,E)$
where $|S|$ is a finite set, $\pq$ a partial ordering on $|S|,$
and $E$ a set of unordered pairs of distinct members of $|S|,$
and study, as a function of $n\geq 0,$ the number of maps
$\varphi:|S|\to\{1,\dots,n\}$ which are both isotone with respect
to the ordering $\pq,$ and have the property that
$\varphi(x)\neq \varphi(y)$ whenever $\{x,y\}\in E.$
We prove a number-theoretic result about this function,
and use it in section~\ref{S.GPH} to recover a
ring-theoretic identity of G.\,P.\,Hochschild.

In section~\ref{S.sign} we generalize a result of
R.\,Stanley on the sign-imbalance of posets in which the
lengths of all maximal chains have the same parity.

In sections~\ref{S.L_pm}-\ref{S.iteration} we study
the linearization-count and sign-imbalance
of a lexicographic sum of $n$ finite posets $P_i$ $(1\leq i\leq n)$
over an $\!n\!$-element poset $P_0.$
We note how to compute these values from the corresponding
counts for the given posets $P_i,$ and for a
lexicographic sum over $P_0$ of chains of lengths $\r{card}(P_i).$
This makes the behavior of lexicographic sums of chains
over a finite poset $P_0$ of interest, and we obtain some
general results on
the linearization-count and sign-imbalance of these objects.
\end{abstract}
\maketitle

This material is far from my areas of expertise.
The referees for two journals were not enthusiastic
about it, so I have decided not to publish it, but I am keeping
it available online in case it should prove of interest to someone.

\section{Denominators of order-chromatic polynomials}\label{S.ord+}
In this section, for $S=(|S|,\pq,E)$
as in the first paragraph of the abstract, we show that
the function of $n$ defined in that sentence is a polynomial with
rational coefficients, and we obtain a bound on the primes dividing
its denominator.
This bound is used, as indicated, in section~\ref{S.GPH}
(an appendix, which depends only on the present section).

Some background:
Recall that the {\em chromatic polynomial} of a finite graph $G$ is the
function $f$ associating to every positive integer $n$ the number of
colorings of the vertices of $G$ with $n$ colors (not all of
which need be used), such that adjacent vertices have different colors.
(We will recover below the fact that $f$ is
indeed a polynomial function.)

Richard Stanley notes in~\cite{RPS70} that if $P=(|P|,\pq)$
is a finite poset
(partially ordered set), then the function  $e$ associating
to every positive integer $n$ the number of isotone
maps $\varphi:P\to\{1,\dots,n\}$ (i.e., maps
$\varphi:|P|\to\{1,\dots,n\}$
such that $x\pq y\implies \varphi(x)\leq \varphi(y))$ is also given by
a polynomial, conceptually similar to the chromatic polynomial
of a graph, and that the same is true of the function
$\overline{e}$ associating to $n$ the number of {\em strictly} isotone
maps $P\to\{1,\dots,n\}$ (maps such that for distinct $x,\,y,$
one has $x\pq y\implies \varphi(x)<\varphi(y)).$

To see the above fact for the function $e,$ let us break the process of
choosing an isotone map $\varphi:P\to\{1,\dots,n\}$ into three steps.
First, decide which elements will fall together under $\varphi.$
Calling the quotient poset arising from these identifications
$P',$ choose, next, the linear order which the embedded image
of $|P'|$ in $\{1,\dots,n\}$ is to have.
Finally, choose a way of embedding the resulting linearly ordered set
into $\{1,\dots,n\}.$
Now if we write
\begin{equation}\begin{minipage}[c]{35pc}\label{d.m=}
$m\ =\ \r{card}(P),$
\end{minipage}\end{equation}
then the number of elements of each
$P'$ will be some $m'\leq m,$ and once a linear ordering of
such a $|P'|$
has been chosen, the number of embeddings of the resulting
linearly ordered set in $\{1,\dots,n\}$ will be the number
of ways of choosing $m'$ out of $n$ elements, which is the
binomial coefficient $\binom{n}{\,m'}=n(n-1)\dots(n-m'+1)/m'\,!\,.$
This is a polynomial in $n$ of degree $m'$ over the rational numbers.
Summing the polynomials arising in this way
from all linearly ordered sets obtained as in the
first two steps above, we get a polynomial of degree $m.$

Now the coefficients of this polynomial can clearly be
written to the common denominator $m!,$ so every prime dividing
the least common denominator of those coefficients is $\leq m.$
But we can get a better bound on such primes if
the Hasse diagram of $P$ has more than one connected component.

To do so, let us modify the procedure described above.
The first step, choosing which elements of $|P|$ fall together,
remains the same; again let $P'$ denote the set so obtained, given
with the weakest partial ordering making the map $P\to P'$ isotone.
(Here we are assuming that $|P'|$ is a set-theoretic image of $|P|$
which admits a partial ordering making the map $|P|\to |P'|$
isotone.
The weakest such partial ordering is then
the intersection of all such partial orderings.)
At the second step, rather than strengthening that ordering
to a linear ordering on $|P'|,$ let us strengthen it to an
ordering which is linear on each connected component of the
Hasse diagram of $P',$
but leaves elements of distinct connected components incomparable.
Finally, we choose a one-to-one isotone map of the
resulting poset into $\{1,\dots,n\}.$

How many choices are there at that last step?
If, under the ordering chosen at the second step,
the connected components of $P'$ are chains $P'_1,\dots,P'_k,$ of
cardinalities $m'_1,\dots,m'_k$ respectively, then there are
$\binom{n}{\,m'_1}$ isotone embeddings of $P'_1$ in $\{1,\dots,n\},$
then $\binom{\,n-m'_1}{m'_2}$ such embeddings of $P'_2$ in
the remaining elements of $\{1,\dots,n\},$ and so on.
So the total number of
isotone embeddings is $\binom{n\,}{\,m'_1}\binom{\,n-m'_1}{m'_2}\dots
\binom{\,n-m'_1-\ldots-m'_{k-1}}{m'_k},$ which,
writing $m'_1+\dots+m'_k = m' = \r{card}(|P'|),$ is
the multinomial coefficient
\begin{equation}\begin{minipage}[c]{35pc}\label{d.multinom}
$n(n-1)\dots (n-m'+1)/m_1!\,\dots m_k!\ .$
\end{minipage}\end{equation}
This is a polynomial in $n$ whose denominator is divisible only
by primes that are less than or equal to one of the $m'_i,$
so we need to know how large the $m'_i,$ the cardinalities of
the components of $P',$ can be.

Suppose the original poset $P$ had $c$ connected components.
Each identification of two elements of $P$ reduces
the number of connected components by at most one, hence
if $P'$ has $m'$ elements, and so has undergone
$m-m'$ identifications, it still has at least $c-(m-m')$ components.
The largest that such a component can be is $m'$ minus
the number of {\em other} components,
so each component has at most $m'-(c-(m-m')-1)=m-c+1$ elements.

Thus, $m-c+1$ is an upper bound on the primes that can
occur in the denominator of~\eqref{d.multinom}.
Summing over all choices of our image $P'$ of $P,$
and all componentwise linearizations thereof, we conclude
that the polynomial giving the number of isotone maps
$P\to\{1,\dots,n\}$ has denominator divisible only by
primes less than or equal to $m-c+1.$

As noted in the abstract, we actually want to count
the smaller set of such maps $\varphi:P\to \{1,\dots,n\}$
that satisfy an additional set of restrictions on which pairs of
elements of $P$ can fall together under $\varphi.$
But such restrictions do not affect the above
argument: they limit the set of images $P'$ that we
enumerate at the first step, while our bound on primes
in the denominator comes from the last step.
So we have

\begin{theorem}\label{T.ord+}
Let $S$ be a $\!3\!$-tuple $(|S|,\pq,E),$ where $|S|$ is a
finite set, $\pq$ a partial ordering on $|S|,$
and $E$ a set of unordered pairs of distinct members of $|S|.$
Let $c$ be the number
of connected components of the Hasse diagram of the poset $(|S|,\pq),$
and for each $n\geq 0,$ let $C(S,n)$ be the number of maps
$\varphi:|S|\to\{1,\dots,n\}$ which are both isotone with respect
to the ordering $\pq,$ and have the property that
$\varphi(x)\neq \varphi(y)$ whenever $\{x,y\}\in E.$

Then $C(S,n)$ is a polynomial in $n$ with rational coefficients,
such that all primes dividing the denominators
of those coefficients are $\leq \r{card}(|S|)-c+1.$\qed
\end{theorem}

Note that in the case of the above result where
$\pq$ is the trivial ordering, i.e., where $S$ is an antichain,
$C(S,n)$ is the chromatic polynomial of the graph $(|S|,E).$
In that case $c=\r{card}(|S|),$ so the theorem says $C(S,n)$ is
a polynomial with {\em integer} coefficients; which is indeed
true of chromatic polynomials \cite[\S IX.2]{tutte}.

Theorem~\ref{T.ord+} actually remains true if we allow arbitrary
Boolean conditions on which elements fall together.
For instance, for particular $x_1,$ $x_2,$ $y_1,$ $y_2\in |S|,$
the condition, ``if $x_1$ falls together with $y_1,$ then
$x_2$ falls together with $y_2$'' would cause no difficulties
with the proof.
I have assumed the conditions to have the simple form stated in
the theorem because such conditions seem to occur naturally;
in particular, the problem that motivated this result,
studied in section~\ref{S.GPH} below, has conditions of that form.
The result for arbitrary Boolean expressions in such
conditions is, in any case, easily deduced from
Theorem~\ref{T.ord+} by inclusion-exclusion considerations.

On the other hand, Theorem~\ref{T.ord+} does not remain true if we
allow our additional restrictions to involve the {\em order}
relations among elements $\varphi(x);$
for instance, if for some particular $x,y\in |S|$ we
consider only maps $\varphi$ such that $\varphi(x)\leq \varphi(y).$
To impose such a condition is equivalent to
replacing the $P$ of our introductory
discussion by a poset with its order relation $\pq$ strengthened;
and this can fuse distinct connected components without
reducing the number of elements.
So, for instance, if $(|S|,\pq)$ is an $\!m\!$-element antichain,
then an obvious set of conditions of the above form restricts
us to maps $|S|\to\{1,\dots,n\}$ which are isotone
with respect to a particular total ordering $\pq'$ of $|S|,$
and the number of such maps is $\binom{n+m-1}{m},$
a polynomial with denominator $m!\,,$ despite the fact
that the original antichain ordering $\pq$ satisfies $m-c+1=1.$

The following consequence of Theorem~\ref{T.ord+} will be
used in section~\ref{S.GPH}.
Recall that the chromatic number of a graph is the least
positive integer $n$ at which the chromatic polynomial
of the graph is nonzero.

\begin{corollary}\label{C.==0}
Let $S=(|S|,\pq,E)$ and $c$ be as in Theorem~\ref{T.ord+}, and
$p$ be any prime greater than $\r{card}(|S|)-c+1.$

Then for every positive integer $n$ whose residue modulo $p$ is less
than the chromatic number of the graph $(|S|,E),$
the integer $C(S,n)$ is divisible by $p.$
In particular, if $|S|$ is nonempty, then $C(S,n)$ is divisible by~$p$
whenever $n$ is divisible by $p.$
\end{corollary}

\begin{proof}
By Theorem~\ref{T.ord+} we may
write $C(S,n)=f(n)/r,$ where $f$ is a polynomial with
integer coefficients, and $r$ an integer not divisible by $p.$
The residue of $f(n)$ modulo~$p$ depends only on the residue
of $n$ modulo~$p,$ hence as $r$ is invertible modulo~$p,$
the same is true of the residue of $C(S,n)=f(n)/r.$
Now for positive $n_0$ less than the chromatic number of
$(|S|,E),$ there are no maps of $|S|$ into $\{1,\dots,n_0\}$
which distinguish $x$ and $y$ whenever $\{x,y\}\in E,$ in
particular, no isotone maps with that property; hence
$C(S,n_0)=0$ for such $n_0.$
Hence for $n$ congruent to a value $n_0$ in that range, we have
$C(S,n)\equiv_{\pmod{p}} C(S,n_0) = 0.$

The final sentence follows because every nonempty
graph has chromatic number greater than zero.
\end{proof}

The remaining sections of this paper, other than the
appendix section~\ref{S.GPH}, are related to this one only
in that they involve counting linear images of posets, and use
somewhat similar methods.

A belated remark on notation:  At the end of the third
paragraph of this section, I spoke of the condition that ``for
distinct $x,\,y,$ one has $x\pq y\implies \varphi(x)<\varphi(y)$'',
where it would be more natural to say that for all $x,$ $y,$
one has $x\prec y\implies \varphi(x)<\varphi(y).$
But I am avoiding the use of ``$\prec$'' for strict
$\!\pq\!$-inequality because many authors
use $x\prec y$ for the condition that $y$ {\em covers} $x$
with respect to a partial order $\leq,$ i.e., that
$x<y$ and there is no element $z$ with $x<z<y.$
At a few places below, we will indeed consider the covering relation,
though we will not introduce any notation for it.
However, I felt it best to avoid confusion with a notation
commonly used by others.

\section{The sign-imbalance of a bicolorable poset}\label{S.sign}

If $X$ is a set of $n<\infty$ elements, then there are precisely $n!$
bijections $X\to\{1,\dots,n\},$ and given one such
bijection $b,$ we can write any other as $\pi\,b$
for a unique permutation $\pi$ of $\{1,\dots,n\}.$
In this situation, one calls a bijection
$X\to\{1,\dots,n\}$ ``even'' or ``odd''
(relative to $b),$ or equivalently, of sign $+1$ or $-1,$
according to whether $\pi$ is an even or an odd permutation.
Clearly, replacing $b$ with another bijection $X\to\{1,\dots,n\}$
either preserves or reverses this
classification of bijections.

If $P=(|P|,\pq)$ is an $\!n\!$-element poset, then the
{\em isotone} bijections $|P|\to\{1,\dots,n\}$ correspond naturally
to linearizations of $P$ (total orderings refining $\pq),$
and one calls a linearization even or odd if the corresponding
bijection is so.
The number of linearizations that are even minus the number
that are odd is called the {\em sign-imbalance} of the
poset $P$ (see \cite{RPS05} and papers referenced there).
If that number is zero, $P$ is called {\em sign-balanced}.
The sign-imbalance is, of course, unique only up to sign,
unless a particular ordering $b$ relative to which it is calculated
has been specified.

In studying the sign-imbalance of a finite poset $P,$
our method will again be to break the choice of a linearization
of $P$ into steps, such that for each partial ordering achieved at the
next-to-last step, the number of choices available at the final step
is easy to describe.

Since only one-to-one images of $P$ are considered, we have no
use for a graph structure $E$ on $|P|$ restricting which
elements fall together.
We will, however, make use of the graph structure of the
Hasse diagram of $P,$ via

\begin{definition}\label{D.bicolor}
A {\em bicoloring} of a finite poset $P=(|P|,\pq)$ will mean
a coloring of its Hasse graph with two colors; that is,
a partition of $|P|$ into two subsets, such that whenever one element
of $|P|$ covers another, the elements
belong to different members of the partition
\textup{(}``have opposite colors''\textup{)}.
\end{definition}

Thus, $P$ admits a bicoloring if and only if every cycle
in its Hasse diagram has an even number of vertices.

If $P=(|P|,\pq)$ is a finite poset on which we have a bicoloring,
and $\leq$ is some total ordering
refining the given partial ordering $\pq,$ let us break $|P|$
into the ``blocks'' of elements of the same color that
are consecutive under ${\leq}\,.$
That is, let $|P|_1$ be the subset of $|P|$ consisting of the least
element under ${\leq}\,,$ and all elements of the same color (if
any) that follow it under
$\leq$ without an intermediate element of the
opposite color; let $|P|_2$ consist of the first element of
the other color, and those that similarly follow it,
$|P|_3$ the next block (which has the same color as $|P|_1),$ and so on.
Note that within each $|P|_i,$ all elements are incomparable
under $\pq,$ since by the definition of a bicoloring,
any two comparable elements of the same color have
an element of the opposite color between them under ${\pq}\,,$
and hence under $\leq.$

In the above situation, we can define an intermediate
partial ordering $\leq_\r{pre}$ on $|P|,$ under
which members of each $|P|_i$ are incomparable,
as they are under $\pq,$ while
whenever $x\in |P|_i$ and $y\in |P|_j$ with $i<j,$
we let $x\leq_\r{pre} y.$
We shall call the partial ordering $\leq_\r{pre}$
the ``prelinearization'' of $\pq$
determined by the linearization ${\leq}\,.$

Let us call linearizations $\leq$ and $\leq'$ of $\pq$
``block-equivalent'' if the prelinearizations
$\leq_\r{pre}$ and $\leq'_\r{pre}$ are the same.
Given a prelinearization of $P,$ with block decomposition
$|P| = |P|_1\cup |P|_2\cup\dots\cup |P|_m,$ the
set of linearizations $\leq$ of $\pq$
that determine that prelinearization (a
block-equivalence-class of linearizations of $\pq)$ will have
exactly $\r{card}(|P|_1)!\ \r{card}(|P|_2)!\,\dots\,\r{card}(|P|_m)!$
elements, every such linearization being obtained by choosing
an arbitrary linearization of each $|P|_i.$
The sign of each such linearization of $P$ will be, up to a fixed
factor $\pm 1,$
the product of the signs of the linearizations of these subsets.

Now if any $|P|_i$ has cardinality greater than $1,$
then the $\r{card}(|P|_i)!$ choices for the ordering of that
block will be half even and half odd, so the parities of the
linearizations in the whole block-equivalence class will
be equally split between odd and even.
Hence in computing the summation of $\!+1\!$'s and $\!-1\!$'s that
gives the sign-imbalance of $P,$ it suffices
to look at linearizations in which each same-color block
is a singleton; in other words, linearizations for
which the given coloring of $P$ is still a bicoloring.
Let us call such a linearization {\em compatible} with the given
bicoloring.
Then the above discussion yields

\begin{theorem}\label{T.bicolor}
If $P=(|P|,\pq)$ is a finite bicolored poset, then the sign-imbalance
of the set of all linearizations of $P$ \textup{(}relative to any
fixed indexing of $|P|)$ is equal to the sign-imbalance
\textup{(}relative to the same indexing\textup{)} of the set of
linearizations compatible with the given bicoloring.

In particular, if $P$ has no linearizations
compatible with the given bicoloring, then it is sign-balanced.\qed
\end{theorem}

Easy examples of bicolored posets $P$ with no
compatible linearizations are those in which the difference
between the numbers of elements of the two colors is greater than $1.$
If a poset $P$ has only one connected component, then a bicoloring
of $P,$ if one exists,
is unique up to interchange of colors; an example in which
this essentially unique bicoloring has the property just mentioned is
\raisebox{1.5pt}[7pt][7pt]{ 
\begin{picture}(16,12)
\put(7,7){\circle*{2}}
\dotline{0,-4}{7,7}
\dotline{7,-4}{7,7}
\dotline{14,-4}{7,7}
\end{picture}}\,,
with three elements of one color and just one of the other.
So that poset is sign-balanced.

If a bicolorable poset has more than one connected component,
we have a larger choice of bicolorings.
For instance, if $P$ consists of two connected components, each
of which is a
chain with an odd number of elements, and we bicolor those chains
so that their least elements have the same color, then we find
that the difference between the total numbers of elements of the two
colors is $2;$ so this poset, too, is sign-balanced.
If instead we color the two chains so that their least elements
have opposite colors, then $P$
has the same number of elements of each color;
so though we have just seen that it is sign-balanced, this second
bicoloring does not give a proof of that fact.

Examples of {\em non}-sign-balanced posets are
\raisebox{1.5pt}[7pt][7pt]{ 
\begin{picture}(12,12)
\put(0,1){\circle*{2}}
\put(8,6){\circle*{2}}
\dotline{8,-4}{8,6}
\end{picture}}
and
\raisebox{1.5pt}[7pt][7pt]{ 
\begin{picture}(10,12)
\put(0,6){\circle*{2}}
\dotline{0,-4}{0,6}
\put(8,6){\circle*{2}}
\dotline{8,-4}{8,6}
\end{picture}}\,;
as one can quickly check by counting.

Theorem~\ref{T.bicolor} allows us to give an alternative proof of
the following result of R.\,Stanley.
(The formulation of that result in~\cite{RPS05}
refers to the {\em length} of chains,
i.e., the number of {\em steps}, which is one less than the cardinality
of the chain, in terms of which I state it here.)

\begin{lemma}[{R.\,Stanley~\cite[Corollary~2.2]{RPS05}}]\label{L.RPS}
If a finite poset $P=(|P|,\pq)$ has the property
that the cardinalities of its maximal chains all have the same parity,
and this is the opposite of the parity of the
cardinality of $P,$ then $P$ is sign-balanced.
\end{lemma}

\begin{proof}
The fact that the cardinalities of all maximal chains of $P$ have the
same parity implies in particular that for each $x\in |P|,$ the
cardinalities of all maximal
chains in the downset generated by $x$ have a common parity.
(Otherwise, by combining two such chains of different
parities with a common maximal chain in the upset generated by $x,$
one would get two maximal chains in $P$ of different parities.)
Classifying elements $x$ according to whether maximal chains
below $x$ all have odd or even cardinality, we get a bicoloring
of $P,$ under which all minimal elements have a common color.
From this we see that if the common parity of the cardinalities
of all maximal chains of $P$ is odd, the maximal elements
of $P$ will all be of the same color as the minimal elements,
while if it is even, they will all have the other color.

Suppose now that $P$ has a linearization compatible with the
above bicoloring.
Because the parity of $P$ is the opposite of that of its
maximal chains, the relationship between the colors of the elements
at the bottom and top of this linearization (same-color
versus different-color) will be
the opposite of what is the case for maximal chains of $P.$

On the other hand, the bottom element of the linearization must be
the bottom element of one of those maximal chains, and hence must be of
the color which all those bottom elements have,
and the top element must similarly be
the color of the top elements of all the maximal chains,
so, in contradiction to the preceding paragraph, the
relationship between the colors of elements at
the top and bottom of this linearization must be
the same as for maximal chains of $P.$
This contradiction shows that
no compatible linearization of $P$ can exist; so
by Theorem~\ref{T.bicolor}, $P$ is sign-balanced.
\end{proof}

An easy example to which the above lemma applies is that in which
$P$ has two connected components, each a chain of odd cardinality (for
which we got the same conclusion in the second
paragraph after Theorem~\ref{T.bicolor}).
Two more easy examples are the posets
\raisebox{1.5pt}[7pt][7pt]{ 
\begin{picture}(16,12)
\put(7,7){\circle*{2}}
\dotline{0,0}{7,7}
\dotline{14,0}{7,7}
\end{picture}}
and
\raisebox{1.5pt}[7pt][7pt]{ 
\begin{picture}(16,12)
\dotline{0,0}{7,7}
\dotline{7,7}{14,0}
\dotline{14,0}{7,-7}
\dotline{7,-7}{0,0}
\end{picture}}\,.

\section{Bringing sign-imbalance and linearization-count together}\label{S.L_pm}
If $P$ is a finite poset, let $L_0(P)$ denote the number of
even linearizations of $P,$ and $L_1(P)$ the number of
odd linearizations (relative to some fixed linearization).
Then it is natural
to associate to $P$ the element $L(P) = L_0(P)+L_1(P)\,\zeta$
of the group ring $\Z\{1,\zeta\},$ where $\{1,\zeta\}$ is
the $\!2\!$-element group, written multiplicatively.

$\Z\{1,\zeta\}$ admits two ring homomorphisms
to $\Z,$ taking $\zeta$ to $+1$ and~$-1$ respectively; thus,
the former carries $L(P)$ to $L_0(P)+L_1(P),$ the number
of linearizations of $P,$ which we shall denote $L_+(P),$ while the
latter carries it to $L_0(P)-L_1(P),$ the sign-imbalance of $P,$
which we shall denote $L_-(P).$
These homomorphisms together yield an embedding of $\Z\{1,\zeta\}$
in $\Z\times\Z,$ whose image is
$\{(a,b)\in\Z\times\Z\mid a\equiv b\!\pmod{2}\}.$
The image of $L(P) = L_0(P)+L_1(P)\,\zeta$ under this embedding,
$(L_+(P),L_-(P)),$ will be denoted $L_\pm(P).$

For any finite poset $P,$ the element $L(P)$ lies in the
subsemiring $\N\{1,\zeta\}$ of $\Z\{1,\zeta\}$ consisting of elements
in which the coefficients of $1$ and $\zeta$ are both nonnegative.
The image of this semiring under the above ring embedding~is
\begin{equation}\begin{minipage}[c]{35pc}\label{d.bleq|a|}
$\{(a,b)\in\Z\times\Z\mid\,a\equiv b\hspace{-.5em}\pmod{2}$
and $|b|\leq a\}.$
\end{minipage}\end{equation}

It will follow from Theorem~\ref{T.L(*)} below that for
various finite posets $P$ and $P',$ we can say that the number
of linearizations $L_+(P)$ is a multiple of $L_+(P').$
(For instance, this is true if $P'$ is any of the
connected components of $P.)$
In such situations, it is not surprising to find that
the sign-imbalance $L_-(P)$ is, likewise, a multiple of $L_-(P').$
But in fact, one typically has the stronger statement that
$L(P)$ is a multiple of $L(P')$ in $\N\{1,\zeta\},$
equivalently, that the ordered pair $L_\pm(P)$
is a multiple of $L_\pm(P')$ in the semiring~\eqref{d.bleq|a|}.

So, for instance, if $L_\pm(P')=(4,2),$ then
$L_\pm(P)$ cannot be $(12,0)$ or $(8,2)$ or $(8,8).$
Separate consideration of the numbers $L_+(P)$
and $L_-(P)$ would not exclude these values;
but the pairs of integers by which $(4,2)$ would have to
be multiplied to get them would be $(3,0),$
$(2,1)$ and $(2,4),$ of which the first two don't have
coordinates of the same parity, while the last
fails to satisfy the inequality $|b|\leq a$ of~\eqref{d.bleq|a|}.
(The element $(2,4)\in\Z\times\Z$ corresponds to
$3-\zeta\notin\N\{1,\zeta\}.)$

We will, in the next three sections, consider $L_+$ and $L_-$ together.
We will encounter both parallelisms and contrasts between
their behaviors.

We remark that the element $L(P)\in\N\{1,\zeta\}$ defined above is the
image, under the homomorphism taking~$q$ to~$\zeta,$
of the polynomial $I_{P,\omega}(q)$ defined by Stanley in
\cite[p.881, display~(1)]{RPS05}.
That polynomial depends in a much stronger way than $L(P)$ on the
reference linearization (expressed there by an indexing $\omega$
of $|P|),$ and, as Stanley notes, does not seem easy to understand.

\Needspace{4\baselineskip}
\section{The linearization-count and sign-imbalance of a lexicographic sum}\label{S.lex}

\subsection{Review of lexicographic sums}\label{SS.lex}
Suppose $P_0=(|P_0|,\pq_0)$ is a poset, and that for each $x\in |P_0|$
we are given a poset $P_x=(|P_x|,\pq_x).$
Then we can form the disjoint union of the sets $|P_x|,$ say
constructed as $\{(x,x')\mid x\in |P_0|,\,x'\in |P_x|\},$ and give
this set a partial ordering under which, for each $x\in |P_0|,$ the
copy $\{(x,x')\mid x'\in |P_x|\}$ of $|P_x|$ is made order-isomorphic
to $P_x$ via the correspondence $(x,x')\mapsto x',$
while for $x\neq y$ in $|P_0|,$ the order-relation between
elements $(x,x')$ and $(y,y')$ $(x'\in |P_x|,$ $y'\in |P_y|)$
is determined solely by their first components:
$(x,x')\pq(y,y')$ if and only if $x\pq_0 y.$

For finite posets, considered
below, we shall find it convenient
to assume $|P_0|$ indexed as $\{x_1,\dots,x_{m_0}\}.$
We shall then abbreviate
$P_{x_i}=(|P_{x_i}|,\pq_{x_i})$ to $P_i=(|P_i|,\pq_i),$ and
assume that each $|P_i|$ is indexed as $\{x_{i,1},\dots,x_{i,m_i}\},$
with distinct symbols $x_{i,j}$
$(i\leq m_0,\,j\leq m_i)$ denoting distinct elements.
This makes the $|P_i|$ disjoint, so we can forego the construction
of their disjoint union by ordered pairs,
and simply take the underlying set of our lexicographic sum
to be their union.
Choosing a notation for such a lexicographic sum, and summarizing
the above description of it, we have
\begin{equation}\begin{minipage}[c]{35pc}\label{d.P_0*()}
$P_0*(P_1,\,\dots,\,P_{m_0})\ =\ (|P|,\,\pq)\ =\ P,$ where\\[.2em]
$|P|\ =\ \{x_{i,j}\mid 1\leq i\leq m_0,\ 1\leq j\leq m_i\},$ and\\[.2em]
$x_{i,j}\pq x_{i',j'}$ if and only if either\\[.1em]
\hspace*{1em}$i\neq i'$ and $x_i\pq_0 x_{i'}$ (in $P_0),$ or\\
\hspace*{1em}$i=i'$ and $x_{i,j}\pq_i x_{i,j'}$ (in $P_i).$
\end{minipage}\end{equation}

Two easy classes of examples:
A poset $P$ decomposed into its connected components $P_i$
can be regarded as the lexicographic sum of the $P_i$
over an antichain $P_0;$ in particular, the
next-to-last steps in the construction of section~\ref{S.ord+} were
lexicographic sums of chains over antichains.
On the other hand, the ``prelinearizations'' that occurred at the
next-to-last step in section~\ref{S.sign} were
lexicographic sums of antichains over chains.

In a lexicographic sum $P_0*(P_1,\,\dots,\,P_{m_0}),$ some or all
of the $P_i$ may be empty.
Occasionally, such cases will require special consideration.

Let us make

\begin{convention}\label{C.lin}
In the context of~\eqref{d.P_0*()}, for the purpose of defining the
parities of linearizations, we define on each $|P_i|$
$(0\leq i\leq m_0)$ the linear {\em reference
order} $x_{i,1}\pq_{i,\r{ch}}\dots\pq_{i,\r{ch}}x_{i,m_i}$
\textup{(}where $\r{ch}$ is mnemonic for ``chain''\textup{)},
and for each $i$ we shall write $P_{i,\r{ch}}= (|P_i|,\pq_\r{ch}).$

On $|P_0*(P_1,\,\dots,\,P_{m_0})|$ we likewise
define the reference order
\begin{equation}\begin{minipage}[c]{35pc}\label{d.pq_ch}
$x_{1,1}\pq_\r{ch}\dots\pq_\r{ch}x_{1,m_1}\pq_\r{ch}\ \dots
\ \pq_\r{ch} x_{m_0,1}\pq_\r{ch}\dots \pq_\r{ch}x_{m_0,m_{m_0}}.$
\end{minipage}\end{equation}
\end{convention}

\subsection{Computing $L$ of a lexicographic sum}\label{SS.L(*)}
Sections~\ref{S.ord+} and~\ref{S.sign} both used the idea,
``linearize the pieces, and see how you
can combine the resulting chains''.
This idea is abstracted in the following easy result.

\begin{theorem}\label{T.L(*)}
In the context of~\eqref{d.P_0*()}
and Convention~\ref{C.lin}, we have
\begin{equation}\begin{minipage}[c]{35pc}\label{d.L(*)}
$L(P_0*(P_1,\,\dots,\,P_{m_0}))\ =
\ \left( \prod_{1\leq i\leq m_0} L(P_i)\right)
\,\cdot\,L(P_0*(P_{1,\r{ch}},\dots,P_{m_0,\r{ch}})).$
\end{minipage}\end{equation}

\textup{(}Hence, the same is true with $L$ everywhere
replaced by $L_\pm,$ or by $L_+,$ or by $L_-.)$
\end{theorem}

\begin{proof}[Sketch of proof]
To obtain the general linearization of the
poset $P_0*(P_1,\,\dots,\,P_{m_0})=(|P|,\pq),$
let us first choose the linear order to be used on each
$|P_i|,$ which can be any linearization of ${\pq_i},$
then specify how the union of the resulting chains is to be ordered.
By the definition of lexicographic sum, the order-relation
in $P_0*(P_1,\,\dots,\,P_{m_0})=(|P|,\pq)$
between any element of $P_i$ and any element of $P_j$ with $i\neq j,$
depends only on $i$ and $j;$ hence for each
way of linearizing the separate $P_i,$ the ways of linearizing
their union which are compatible with those linearizations
and with the ordering of $P_0*(P_1,\,\dots,\,P_{m_0})=(|P|,\pq)$
correspond to the ways of linearizing
a lexicographic sum over $P_0$ of a family of chains of the
corresponding lengths, of which
$P_0*(P_{1,\r{ch}},\,\dots,\,P_{m_0,\r{ch}})$ is one.
It is also not hard to check that the parity of a
linearization of $P_0*(P_1,\,\dots,\,P_{m_0})$
(relative to the reference-ordering $\pq_\r{ch})$
is the product of the parities of the induced linearizations
of the $P_i$ (relative to the reference-orderings $\pq_{i,\r{ch}})$
and that of the corresponding linearization
of $P_0*(P_{1,\r{ch}},\,\dots,\,P_{m_0,\r{ch}})$
(again relative to $\pq_\r{ch}).$
The equality~\eqref{d.L(*)} follows.

The final parenthetical assertion follows because the operator
$L_\pm$ is the composite of the operator $L$ with a
ring homomorphism $\Z\{1,\zeta\}\to \Z\times\Z,$ and $L_+,$
$L_-$ are in turn the composites of $L_\pm$ with the projection
homomorphisms $\Z\times\Z\to\Z.$
\end{proof}

\subsection{Results on lexicographic sums of chains}\label{SS.L(m...)}
For Theorem~\ref{T.L(*)}
to be useful, we need to be able to compute $L(P)$
for $P$ a lexicographic sum of {\em chains.}
We obtain results in this direction below.
Since in these results, the chains in question are not given as
listings of the elements of
other posets, we adjust our notation slightly.

\begin{definition}\label{D.ch}
Let $P_0=(|P_0|,\pq_0)$ be a finite poset,
with $|P_0|=\{x_1,\dots,x_{m_0}\}.$
Then for nonnegative integers $m_1,\dots,m_{m_0},$ we define
\begin{equation}\begin{minipage}[c]{35pc}\label{d.LPm}
$L(P_0;\,m_1,\dots,m_{m_0})\ =\ L(P_0*(C_1,\dots,C_{m_0}))$
\end{minipage}\end{equation}
where for $1\leq i\leq m_0,$ $C_i=(|C_i|,\pq_i)$ is a chain of $m_i$
elements $x_{i,1}\pq_i\dots\pq_i x_{i,m_i},$
these chains being understood to be pairwise disjoint,
and the parities of linearizations of their
lexicographic sum being taken relative to the reference
ordering~\eqref{d.pq_ch}

We will also use the notation corresponding to~\eqref{d.LPm} with
$L_0,$ $L_1,$ $L_+,$ $L_-$ and $L_\pm$ in place of $L.$
\end{definition}

A trivial case is

\begin{proposition}\label{P.chain}
Suppose, in Definition~\ref{D.ch},
that $P_0$ is a chain, $x_1\pq_0\dots\pq_0 x_{m_0}.$
Then for all $m_1,\dots,m_{m_0},$
$L(P_0;\,m_1,\dots,m_{m_0}) = 1.$
Equivalently,
$L_+(P_0;\,m_1,\dots,m_{m_0}) = L_-(P_0;\,m_1,\dots,m_{m_0}) = 1.$
\end{proposition}

\begin{proof}
With $P_0$ and all of $C_1,\dots,C_{m_0}$ being chains, we see that
$P_0*(C_1,\dots,C_{m_0})$ is a chain, so it has just one linearization.
This is our reference ordering, so it has even parity, giving our
first conclusion, which is equivalent to the final equation.
\end{proof}

If $P_0$ is an antichain we can also get an exact result.
In the proofs of Proposition~\ref{P.antich} below, and
of Theorem~\ref{T.floor}, which will generalize the hard part of that
result, we shall, following~\cite{RPS05}, use
the imagery of ``dominoes''.
Namely, to partly or completely cover a poset $P$ with dominoes
means to distinguish certain {\em non-overlapping} pairs of
elements $(x,y)$ such that $y$ covers $x$
(i.e., $x$ and $y$ are distinct elements, such that
$x\pq y$ and there are no elements strictly between these).
If $(x,y)$ is a pair which we have so distinguished,
we will speak of there being a domino lying over $x$ and~$y.$

(Richard Stanley has pointed out
to me that Proposition~\ref{P.antich}
is essentially known: Parts~(i) and~(ii) can be obtained by
setting $q=-1$ in formula~(1.68) of
\cite{RPS:EC}, and calling on~(1.66) and~(1.87) thereof.)

\begin{proposition}\label{P.antich}
Suppose, in the context of Definition~\ref{D.ch},
that $P_0$ is an antichain.

Then the linearization count
$L_+(P_0;\,m_1,\dots,m_{m_0})$ is the multinomial coefficient
\begin{equation}\begin{minipage}[c]{35pc}\label{d.L+=multinom}
$(\sum_{1\leq i\leq m_0} m_i)\,!\,/\,\prod_{1\leq i\leq m_0} m_i !\,.$
\end{minipage}\end{equation}

On the other hand, the sign-imbalance
$L_-(P_0;\,m_1,\dots,m_{m_0})$ is described as follows:\\[.2em]
\textup{(i)}\, If at most one of $m_1,\dots,m_{m_0}$ is
odd, then $L_-(P_0;\,m_1,\dots,m_{m_0})$ is the multinomial coefficient
\begin{equation}\begin{minipage}[c]{35pc}\label{d.L-=multinom}
$(\sum_{1\leq i\leq m_0} \lfloor m_i/2\rfloor)\,!\,/\,
\prod_{1\leq i\leq m_0} \lfloor m_i/2\rfloor\,!\,.$
\end{minipage}\end{equation}

\noindent
\textup{(ii)} If more than one of $m_1,\dots,m_{m_0}$ are
odd, then $L_-(P_0;\,m_1,\dots,m_{m_0})=0.$
\end{proposition}

\begin{proof}
The number $L_+(P_0;\,m_1,\dots,m_{m_0})$
of linearizations of our lexicographic sum of chains
is the number of ways of partitioning
the $\sum_{1\leq i\leq m_0} m_i$ positions
available in such a linearization
into subsets of cardinalities $m_1,\dots,m_{m_0},$
and~\eqref{d.L+=multinom} is a standard description of this number
\cite[p.20]{RPS:EC}.

To get information on $L_-(P_0;\,m_1,\dots,m_{m_0}),$ let us cover
as much as we can of
each linearization of $P_0*(C_1,\dots,C_{m_0})$ with
dominoes, starting from the bottom.
Thus, if $\sum m_i$ is even, each linearization of our
lexicographic sum will be entirely
covered, while if it is odd, all but the top element of
each linearization will be covered.

I claim that the set of linearizations which, when dominoes are
so placed, have
the property that at least one domino lies over two elements that
are {\em incomparable} under $\pq$ is sign-balanced.
Indeed, let us pair off these linearizations as follows.
Given such a linearization, find, among dominoes that lie over
$\pq\!$-incomparable pairs, the top one, and form a
new linearization by reversing the positions of the pair of
elements it lies over.
Because those elements are $\pq\!$-incomparable, the result will
again be a linearization of $\pq;$ and because exactly
one pair has been interchanged, the new linear order has
parity opposite to the old one.
It is also clear that the above construction, applied to the new
linearization, returns the original one, so we indeed have a pairing.
Hence the contributions to the sign-imbalance of those linearizations
sum to zero.

Thus, to determine the sign-imbalance of $P_0*(C_1,\dots,C_{m_0}),$
it suffices to consider linearizations which, when dominoes are
set down as above,
have the property that every domino lies over a pair of
{\em comparable} elements.
Let us call such a linearization of $\pq$ {\em distinguished}.

Since elements from different chains $C_i$ are incomparable,
a distinguished linearization will have the property that
every domino lies over a pair of elements from the same $C_i.$
More precisely, we see by induction, working up from the bottom, that
every domino will lie over a pair of the form $(x_{i,2j-1},\,x_{i,2j}).$
From this it easily follows that any distinguished linearization can be
turned into our reference ordering~\eqref{d.pq_ch}
by repeatedly switching the places of two adjacent dominoes (hence,
reordering a string $w<x<y<z$ as $y<z<w<x),$
or moving a domino past the lone dominoless element, if there is one
(reordering a string $x<y<z$ as $z<x<y.)$
A movement of either sort
can be done using $2$ transpositions of elements,
so in each case, the parity of the linearization does not change.
Hence every distinguished linearization is even,
so the sign-imbalance of $P_0*(C_1,\dots,C_{m_0})$
equals the number of distinguished linearizations.

From the above considerations, we can see that if there exists
a distinguished linearization, at most one of the $C_i$
can have odd cardinality (the $C_i$ to which the ``lone element'',
if any, at the top of a distinguished linearization belongs).
This immediately gives statement~(ii).

To get statement~(i),
suppose first that all the $m_i$ are even.
Then it is not hard to see that each distinguished linearization is
determined by noting how the $(\sum m_i)/2$ dominoes
are partitioned into $m_1/2$ dominoes lying over pairs
of members of $C_1,$ $m_2/2$ lying over pairs of members of $C_2,$ etc..
By the same counting principle used in
getting~\eqref{d.L+=multinom}, the number of such partitions is
$(\sum m_i/2)\,!/\,\prod(m_i/2)\,!\,,$
which agrees with~\eqref{d.L-=multinom} in this case.

Now suppose instead that exactly one of the
$m_i,$ say $m_{i_0},$ is odd.
Then in a distinguished linearization,
the one element not under any domino,
namely, the top element of our linearization,
is necessarily the largest element of the chain $C_{i_0}.$
So a distinguished linearization will be determined by
the arrangement of the dominoes covering
the remaining elements, and as before, the number of
possibilities is described by~\eqref{d.L-=multinom}.
\end{proof}

For general $P_0,$ we do not have an exact formula
for $L_+(P_0;\,m_1,\dots,m_{m_0});$ but we can again get a result on
$L_-(P_0;\,m_1,\dots,m_{m_0}),$ in the spirit of
the above proposition.

\begin{theorem}\label{T.floor}
Given $P_0$ as in Definition~\ref{D.ch}, and
nonnegative integers $m_1,\dots,m_{m_0},$ let
\begin{equation}\begin{minipage}[c]{35pc}\label{d.S=}
$S\ =\ \{x_i\mid m_i$ is odd$\}\ \subseteq\ |P_0|.$
\end{minipage}\end{equation}
Then\\[.2em]
\textup{(i)}\, If $S$ forms a chain under $\pq_0$
\textup{(}possibly the empty chain\textup{)}, then we have
\begin{equation}\begin{minipage}[c]{35pc}\label{d.L-=}
$L_-(P_0;\,m_1,\dots,m_{m_0})\ =
\ \pm\,L_+(P_0;\lfloor m_1/2\rfloor,\dots,\lfloor m_{m_0}/2\rfloor),$
\end{minipage}\end{equation}
with the $+$ sign applying if the reference ordering of $P_0$
is isotone on $S.$\\[.2em]
\textup{(ii)}\, If, on the other hand, $S$
contains two elements $x_{i_0}$ and $x_{i_1}$ which are incomparable,
and which {\em are majorized by} exactly
the same set of other elements of $P_0,$ or alternatively, which
{\em majorize} exactly the same set of other elements of $P_0,$ then
\begin{equation}\begin{minipage}[c]{35pc}\label{d.L-=0}
$L_-(P_0;\,m_1,\dots,m_{m_0})\ =\ 0.$
\end{minipage}\end{equation}
\end{theorem}

\begin{proof}
Let us place dominoes on each linearization of
$P_0*(C_1,\dots,C_{m_0}),$ say with ordering $\leq,$ as follows,
working upward from the bottom.

Suppose inductively that we have specified where dominoes
are to be placed on our linearization up to a certain point,
but not all the way to the top.
Let $x_{i,j}$ be the least element such that we have not specified
whether a domino is to be placed on it.
If $x_{i,j}$ is the greatest element of our linearization, then
clearly we can put down no more dominoes; in particular,
$x_{i,j}$ will remain uncovered.
If, rather, $x_{i,j}$ is followed immediately either by another
element from same chain $C_i,$ or by an element which belongs
to a different chain
$C_{i'},$ and is $\!\pq\!$-incomparable with $x_{i,j},$
let us put a domino over this pair of elements.
Finally, if $x_{i,j}$ is followed in our linearization by an element
$x_{i',j'}$ with $i'\neq i$ and $x_i\pq_0 x_{i'},$
we put no domino over $x_{i,j}$ (so that our recursive construction
will continue with $x_{i',j'}$ in the role that $x_{i,j}$ had).

Having placed dominoes in this manner on every linearization of
$P_0*(C_1,\dots,C_{m_0}),$ we note that, as in the proof of
Proposition~\ref{P.antich}, if a linearization has a domino placed
over two incomparable elements, then looking at the highest
domino with this property, and reversing the order of the elements
under it, we get
another linearization, of the opposite parity, on which the
dominoes will have been placed over the same pairs of elements.
This again pairs off linearizations of opposite parities, hence we
can again ignore linearizations so paired in determining
$L_-(P_0;\,m_1,\dots,m_{m_0}).$

Let us again call those linearizations which are not so paired,
i.e., which involve dominoes only over pairs
$(x_{i,j},\,x_{i,j+1})$ of elements
of the same chain $C_i,$ ``distinguished''.

Note that the way we have placed dominoes on our linearizations
insures that if an element $x_{i,j}$ is not under a domino, then
it must be the largest element of the chain $C_i$ to
which it belongs, namely, $x_{i,m_i}.$
(Indeed, otherwise, the larger elements of $C_i$ would lie
above it, so it could not be the top element of the linearization; and
they would lie between it and the elements of any $C_{i'}$ with
$x_i\pq_0 x_{i'},$ so it would not get skipped in the process of
setting down dominoes.)
Clearly, also, if such an element is not the
largest element in the whole
linearization, then the element that follows it must have the form
$x_{i',1}.$
Finally, if our linearization is distinguished, then
all pairs of elements under dominoes must have the form
$(x_{i,j},\,x_{i,j+1});$ and for each $i,$ we again see by induction
that dominoes will cover precisely the pairs $(x_{i,2j-1},x_{i,2j})$
for $1\leq j\leq \lfloor m_i/2\rfloor.$
Thus, the top element $x_{i,m_i}$ of a chain $C_i$ will not lie
under a domino if and only if $m_i$ is odd.

Now in every distinguished linearization of
$P_0*(C_1,\dots,C_{m_0}),$
the $\sum_{1\leq i\leq m_0} \lfloor m_i/2\rfloor$
dominoes must be arranged in one of the
$L_+(P_0;\lfloor m_1/2\rfloor,\dots,\lfloor m_{m_0}/2\rfloor)$ ways
corresponding to linearizations of the lexicographic
sum over $P_0$ of a family of chains of lengths $\lfloor m_i/2\rfloor.$
I claim that if, as in the hypothesis of statement~(i), the set $S$
defined in~\eqref{d.S=} forms a chain in $P_0,$ then
(a)~each of the above arrangements of
dominoes appears in one and only one distinguished
linearization of $P_0*(C_1,\dots,C_{m_0}),$ and
(b)~if the reference ordering of $P_0$ is isotone on $S,$
then all distinguished linearizations have even parity.
This will prove~(i).

To see assertion~(a), assume we are given a linearization of
the indicated family of dominoes, and let us see why there is a
unique way of placing the ``lone'' elements $x_{i,m_i}$
$(x_i\in S)$ among them so as to get a distinguished linearization
of $P.$
Note that since $S$ is a chain in $P_0,$ the {\em relative} ordering
of those lone elements is predetermined.

We first describe the unique position at which the {\em highest}
lone element $x_{i,m_i}$ can appear in such a linearization:
If there is no $x_{i'}$ in $P_0$ with $x_i\pq_0 x_{i'},$
then by our description of distinguished linearizations,
$x_{i,m_i}$ can only appear at the very top; while if there are
such elements $x_{i'},$ they will necessarily belong
to $|P_0| - S,$ and $x_{i,m_i}$ will necessarily appear
just below the element $x_{i',1}$ with $x_i\pq_0 x_{i'}$
that appears lowest among such elements in our given
domino-covered chain.

We now repeat this idea, working down the chain of
elements $x_i\in S:$ for each $x_i,$ we place
$x_{i,m_i}$ just below the lowest element $x_{i',1}$
such that $x_i\pq_0 x_{i'}$ that we have already positioned.
(This $x_{i',1}$ may or may not lie under a domino.
It will not if $x_{i'}\in S$ and $m_{i'}$ happens to be $1;$
but in this case, by the order in which we have chosen to do things,
it will be an element whose position we have already determined.)
It is easy to see that the resulting ordering is indeed
a linearization of $P_0*(C_1,\dots,C_{m_0}),$ in fact
a distinguished linearization, and is the only distinguished
linearization compatible with the given arrangement of dominoes.

Assuming now that the reference ordering of $P_0$ is
isotone on $S,$ we see that each of the distinguished linearizations
described above can be obtained from the reference linearization
of $P_0*(C_1,\dots,C_{m_0})$
by a series of steps, each of which either moves one domino past
another, or moves a domino past a lone element.
(By our assumption that the reference ordering of $P_0$ is
isotone on $S,$ we never have to move one lone element past another.)
Again, each of these
steps acts by an even permutation, so our distinguished orderings
indeed have even parity, proving~(b), and hence proving~(i) in
the case where the reference ordering is isotone on $S.$
In the contrary case, a change in the reference ordering will
multiply $L_-(P_0;\,m_1,\dots,m_{m_0})$ by $\pm 1,$ giving
the general case of~(i).

Turning to~(ii), suppose first that $S$
contains incomparable elements $x_{i_0}$ and $x_{i_1}$ which
are majorized by the same sets of other elements of $P_0.$
Then in a distinguished linearization of
$P_0*(C_1,\dots,C_{m_0}),$ the descriptions of where the
elements $x_{i_0,m_{i_0}}$ and $x_{i_1,m_{i_1}}$ must
occur are the same -- at the top if the set of elements
$x_i\in |P_0|$ strictly greater than $x_{i_0}$
(equivalently, strictly greater than $x_{i_1})$
is empty; otherwise, immediately below the lowest element
$x_{i,1}$ of the linearization such that $x_i$ which is
strictly greater than these elements.
But $x_{i_0,m_{i_0}}$ and $x_{i_1,m_{i_1}}$ cannot both be
in that position, so $P_0*(C_1,\dots,C_{m_0})$ has {\em no}
distinguished linearizations, and~\eqref{d.L-=0} follows.

If, rather, $S$ has incomparable elements $x_{i_0}$ and $x_{i_1}$
which {\em majorize} the same set of other elements of $P_0,$
the result follows from the above case by the
invariance of the property of sign-balance under
reversal of order.
This completes the proof of~(ii).

(Since our rule for placing dominoes, and hence
our concept of distinguished linearization, are
{\em not} invariant under reversal of order,
we can't claim in the last case
that $P_0*(C_1,\dots,C_{m_0})$ has no distinguished linearizations.
Rather, a direct proof of~\eqref{d.L-=0} for that case
would involve defining ``reverse-distinguished''
linearizations, and showing that
$P_0*(C_1,\dots,C_{m_0})$ has none of these.)
\end{proof}

If we apply the above theorem to the case where $P_0$ is an
antichain, then every subset $S\subseteq |P_0|$ clearly
falls under one of cases~(i) or~(ii) above,
depending on whether it has $\leq 1$ or $>1$ elements,
and we get the description
of $L_-(P_0;\,m_1,\dots,m_{m_0})$ in Proposition~\ref{P.antich}.
The simplest examples of posets $P_0$
for which not every subset $S$ is covered
by our theorem are the $\!4\!$-element posets
\raisebox{1.5pt}[7pt][7pt]{ 
\begin{picture}(10,12)
\put(0,6){\circle*{2}}
\dotline{0,-4}{0,6}
\put(8,6){\circle*{2}}
\dotline{8,-4}{8,6}
\end{picture}}
and
\raisebox{0.5pt}[7pt][5pt]{ 
\begin{picture}(15,12)
\dotline{0,-3}{5,8}
\dotline{5,8}{10,-3}
\dotline{10,-3}{15,8}
\put(15,8){\circle*{2}}
\end{picture}},
with $S$ consisting, in each case, of the lower left and
upper right elements.
In the former case,
the function $L_\pm(P_0;\,m_1,m_2,m_3,m_4)$ is nevertheless
easily evaluated.
The second case looks harder, but I have not examined it closely.

\section{Some properties of $L_\pm(P_0;\,m_1,\dots,m_{m_0})$ as a function of $m_0,\dots,m_{m_0}$}\label{S.lex+}

\subsection{Motivation: the case where $P_0$ is a $\!2\!$-element antichain}\label{SS.Pascal}
Let us consider the simplest nontrivial case
of Proposition~\ref{P.antich}, where $m_0=2.$
Then $P_0*(C_1,C_2)$ is the disconnected union of a chain
of $m_1$ elements and a chain of $m_2$ elements, so its
linearizations correspond to the ways
of partitioning a chain of $m_1+m_2$ elements into two sets,
of $m_1$ and $m_2$ elements respectively.
The number of these is the binomial
coefficient $\binom{\,m_1+m_2}{\ m_2},$ so the values of
$L_+(P_0;\,m_1,m_2)$ are the entries of Pascal's triangle.
We display in~\eqref{d.Pascal} the first few rows of that triangle,
and likewise,
the values of $L_-(P_0;\,m_1,m_2)$ given by that same proposition.
In each array in~\eqref{d.Pascal}, the rows correspond to the
values of $m_1+m_2,$ the diagonals going downward to the
left to the values of $m_1,$
and the diagonals going downward to the right to the values of~$m_2.$

\begin{equation}\begin{minipage}[c]{35pc}\label{d.Pascal}
\BB{$L_+(P_0;\,m_1,m_2)$}\BB{$L_-(P_0;\,m_1,m_2)$}\\[.4em]
\BB{\B{1}}%
\BB{\B{1}}\\
\BB{\B{1}\B{1}}%
\BB{\B{1}\B{1}}\\
\BB{\B{1}\B{2}\B{1}}%
\BB{\B{1}\B{0}\B{1}}\\
\BB{\B{1}\B{3}\B{3}\B{1}}%
\BB{\B{1}\B{1}\B{1}\B{1}}\\
\BB{\B{1}\B{4}\B{6}\B{4}\B{1}}%
\BB{\B{1}\B{0}\B{2}\B{0}\B{1}}\\
\BB{\B{1}\B{5}\B{10}\B{10}\B{5}\B{1}}%
\BB{\B{1}\B{1}\B{2}\B{2}\B{1}\B{1}}\\
\BB{\B{1}\B{6}\B{15}\B{20}\B{15}\B{6}\B{1}}%
\BB{\B{1}\B{0}\B{3}\B{0}\B{3}\B{0}\B{1}}\\
\BB{\B{1}\B{7}\B{21}\B{35}\B{35}\B{21}\B{7}\B{1}}%
\BB{\B{1}\B{1}\B{3}\B{3}\B{3}\B{3}\B{1}\B{1}}\\
\BB{\B{1}\B{8}\B{28}\B{56}\B{70}\B{56}\B{28}\B{8}\B{1}}%
\BB{\B{1}\B{0}\B{4}\B{0}\B{6}\B{0}\B{4}\B{0}\B{1}}
\end{minipage}\end{equation}
\vspace{.5em}

The familiar rule for producing Pascal's triangle, that each
entry is the sum of the two above it in the preceding row,
can be interpreted in terms of $L_+(P_0;\,m_1,m_2):$
If we classify linearizations of $P_0*(C_1,C_2)$
according to whether the top element belongs
to $C_1$ or $C_2,$ then in the former case,
the ordering of the remaining elements constitutes
a linearization of the union of a chain of $m_1-1$
elements and a chain of $m_2$ elements,
in the latter case, a linearization of a union of a chain of $m_1$
elements and a chain of $m_2-1$ elements.

The sign-imbalances $L_-(P_0;\,m_1,m_2)$ turn out
to satisfy a similar law:
every entry that lies an {\em even} number of
steps from the left-sloping
edge of the array is, as in Pascal's triangle, the sum of the two
entries above it, while if an entry lies an odd number of steps from
that edge, it is the {\em difference} of those two entries
(the one to the left minus the one to the right).
The reader can justify this rule by classifying linearizations
of our union of chains as in the preceding paragraph, and
examining how the parity of a linearization compares with the
parity of the linearization of the one-element-smaller
poset that we get on dropping the element at the top
of that linearization.

(In fact, I discovered the formulas of
Proposition~\ref{P.antich}(i)-(ii) by first studying the case
$m_0=2,$ and obtaining the above recursive
rule for $L_-(P_0;\,m_1,m_2),$ then noticing
the way values from Pascal's triangle occurred in the resulting array,
and thinking about how to justify that pattern.)

\subsection{$L_\pm(P_0;\,m_1,\dots,m_{m_0})$ and one-variable polynomials}\label{SS.poly}

In general, $L_\pm(P_0;\,m_1,\dots,m_{m_0})$ is not
a polynomial function of $m_1,\dots,m_{m_0}.$
For instance, if $P_0$ is a $\!2\!$-element antichain
as in the preceding subsection, we have
$L_+(P_0;m,m) = \binom{2m}{m}\geq 2^m,$
so it grows too fast to be polynomial.
But note also that if, in
$L_+(P_0;\,m_1,m_2) = \binom{\,m_1+m_2}{\ m_1},$
we hold $m_2$ constant, the resulting function
is a polynomial in $m_1$ (of degree $m_2),$ while if
we hold $m_1$ constant, it is a polynomial in $m_2$ (of degree $m_1).$
Here is a similar result about $L_+$ for general $P_0,$
and a somewhat more complicated statement about $L_-.$

\begin{proposition}\label{P.poly}
Let $P_0=(|P_0|,\pq_0)$ be a finite poset, with $|P_0|=
\{x_1,\dots,x_{m_0}\},$ let $i_0\in\{1,\dots,m_0\},$
and for each $i\neq i_0$ in $\{1,\dots,m_0\},$ let
us fix a value for $m_i.$
Then, regarding $L_\pm(P_0;\,m_1,\dots,m_{m_0})$ as a
function of $m_{i_0},$ we have\\[0.4em]
\textup{(a)} $L_+(P_0;\,m_1,\dots,m_{m_0})$ is
a polynomial in $m_{i_0},$ whose degree is the sum of those
values of $m_i$ $(i\neq i_0)$ such that $x_i$ is $\pq_0\!$-incomparable
with $x_{i_0}.$\\[0.4em]
\textup{(b)} The values of $L_-(P_0;\,m_1,\dots,m_{m_0})$ are given by
two polynomials in $m_{i_0},$ one for
$m_{i_0}$ even, the other for $m_{i_0}$ odd.
Each has degree less than or equal to
$\lfloor \sum m_i/2\rfloor,$ where the summation is again over those
$i$ such that $x_i$ is $\pq_0\!$-incomparable with $x_{i_0}.$
\end{proposition}

\begin{proof}
Again let $C_1,\dots,C_{m_0}$ be disjoint
chains of lengths $m_1,\dots,m_{m_0}.$
Then a linearization of $P=P_0*(C_1,\dots,C_{m_0})$ can
be determined by first choosing an arbitrary linearization of
the subposet with underlying set $|P|-|C_{i_0}|,$
then specifying where to insert the elements of $|C_{i_0}|.$

To see where those elements can go, let us
partition $\{1,\dots,m_0\}-\{i_0\}$ into three subsets:
the set $I_<$ of those $i$ such that $x_i\pq_0 x_{i_0},$
the set $I_>$ of those $i$ such that $x_{i_0}\pq_0 x_i,$ and
the set $I_\sim$ of those $i$ such that $x_i$ is
$\pq_0\!$-incomparable with $x_{i_0}.$
Then given any linearization $\leq$ of $\pq$ on $|P|-|C_{i_0}|,$ the
elements of $|C_{i_0}|$ can be inserted within the range bounded
below by the highest location, under that linearization,
of an element of $\bigcup_{i\in I_<} |C_i|,$
and above by the lowest location
of an element of $\bigcup_{i\in I_>} |C_i|;$
where we understand the former restriction to be vacuous
if $\bigcup_{i\in I_<} |C_i|$ is empty,
and the latter if $\bigcup_{i\in I_>} |C_i|$ is empty.

In our linearization of $|P|-|C_{i_0}|,$
the interval we have described will be populated by some subset
(possibly empty) of $\bigcup_{i\in I_\sim} |C_i|.$
If $d$ is the number of elements of that set in that interval, then
clearly
\begin{equation}\begin{minipage}[c]{35pc}\label{d.leq_d_leq}
$0\ \leq\ d\ \leq\ \sum_{i\in I_\sim} m_i,$
\end{minipage}\end{equation}
and we see that the number of ways the $m_{i_0}$ elements
of the chain $C_{i_0}$ can be interspersed among those $d$ elements
is $\binom{\,m_{i_0}+d}{\!d},$ which, as a function
of $m_{i_0},$ is a polynomial of degree $d.$
The maximum value allowed by~\eqref{d.leq_d_leq},
$d=\sum_{i\in I_\sim} m_i,$
does in fact occur, since we can linearize $|P|-|C_{i_0}|$ so that
all members of $\bigcup_{i\in I_<} |C_i|$ precede
all members of $\bigcup_{i\in I_\sim} |C_i|,$ and these precede
all members of $\bigcup_{i\in I_>} |C_i|.$
Summing the polynomials in $m_{i_0}$ obtained from all our
linearizations of $|P|-|C_{i_0}|,$ we get a polynomial $f(m_{i_0})$
describing $L_+(P_0;\,m_1,\dots,m_{d}).$
Since the leading coefficients of the polynomials we have summed
are all positive, the terms
of degree $\sum_{i\in I_\sim} m_i$ cannot cancel, so that is the
degree of $f(m_{i_0}),$ completing the proof of~(a).

If, instead, we look at the sign-imbalance, then the
polynomials $\binom{\,m_{i_0}+d}{\!d}$ are replaced by the
functions described in parts~(i) and~(ii) of
Proposition~\ref{P.antich} for $m_0=2$ (the
diagonals of the right-hand array of~\eqref{d.Pascal}),
multiplied by $\pm 1$ depending on the details of
our linearization of $|P|-|C_{i_0}|.$
Each of these is given by one polynomial of degree $\leq d/2$
(possibly the zero polynomial)
on odd inputs and another polynomial of degree $\leq d/2$
on even inputs, yielding~(b).
Because of the varying signs, we cannot say in this
case that the leading terms of the highest-degree polynomials
will not cancel, hence we cannot specify the exact degrees
of the polynomials we get.
\end{proof}

Comparing with Proposition~\ref{P.antich}, we might wonder whether,
of the two polynomials referred to in~(b) above, the
one that gives the sign-imbalance for $m_{i_0}$ odd must
have degree less than or equal to that
of the one that does so for $m_{i_0}$ even.
But this is not the case.
For instance, suppose $P_0$ is the poset
\raisebox{0.5pt}[7pt][6pt]{ 
\begin{picture}(35,12)
\put(12,2){\circle*{2}}
\dotline{20,-3}{20,8}
\put(20,8){\circle*{2}}
\put(0,0){$x_1$}
\put(21,-5){$x_2$}
\put(21,5){$x_3$}
\end{picture}},
let $i_0=2,$ and take $m_1$ and $m_3$ odd.
Then for any value of $m_2,$ the union of the posets $C_2$ and $C_3$
will be a chain of length $m_2+m_3,$ which has the
opposite of the parity of $m_2,$ whence we see
from Proposition~\ref{P.antich} that for $m_2$ even,
$L_-(P_0;\,m_1,m_2,m_3)$ will be zero, while for $m_2$
odd, it will be a polynomial of degree $\lfloor m_1/2\rfloor.$

\subsection{Chains in $P_0$ again}\label{SS.chains}
At the beginning of the preceding subsection, we noted that
for $P_0$ a $\!2\!$-element antichain, the function $L_+(P_0;\,m_1,m_2)$
was not a polynomial in its two variables.
More generally, if $x_{i_1}$ and $x_{i_2}$ are incomparable
elements of a finite poset $P_0,$ and we fix values for all
$m_i$ other than $m_{i_1}$ and $m_{i_2},$
we find that as a function of those two variables,
$L_+(P_0;\,m_1,\dots,m_{m_0})$ cannot be a polynomial.
Indeed, one can strengthen the ordering of $P_0$
to get an ``almost-chain'' $P'_0,$ in which the only pair of
incomparable elements is $\{x_{i_1},\,x_{i_2}\}.$
Then $L_+(P_0;\,m_1,\dots,m_{m_0})\geq
L_+(P'_0;\,m_1,\dots,m_{m_0})= \binom{m_{i_1}+m_{i_2}}{m_{i_1}}.$
But again, for $m_{i_1}=m_{i_2}=m$ we have
$\binom{m_{i_1}+m_{i_2}}{m_{i_1}}= \binom{2m}{m}\geq 2^m,$
which grows too rapidly for a polynomial function.
It follows that
$L_+(P_0;\,m_1,\dots,m_{m_0})$ is not a polynomial
function in any subset of its arguments two of which correspond
to incomparable elements of~$P_0.$

Can $L_+(P_0;\,m_1,\dots,m_{m_0})$ ever
be a polynomial function of more than one of its arguments?
Yes -- in precisely the cases not excluded by the above observations:

\begin{theorem}\label{T.chains}
Let $P_0=(|P_0|,\pq_0)$ be a finite poset, with $|P_0|=
\{x_1,\dots,x_{m_0}\},$ and $S$ any subset of $|P_0|.$
Suppose that in $L_+(P_0;\,m_1,\dots,m_{m_0})$ we fix nonnegative
values for all the $m_i$ such that $x_i\notin S.$
Then \textup{(}regardless of the values so chosen\textup{)}
the resulting function of the remaining variables \textup{(}the $m_i$
with $x_i\in S)$ is given by a polynomial if and only if
$S$ is a chain in $P_0.$
\end{theorem}

\begin{proof}[Sketch of proof]
We have just seen ``only if''; I shall sketch the proof of ``if''.

Alongside the fixed values for the $m_i$ with $x_i\notin S,$
let us choose arbitrary values for the $m_i$ with $x_i\in S.$
Now let $Q$ denote the lexicographic sum over $(|P_0|-S,\pq_0)$ of
the chains $C_i$ $(x_i\in |P_0|-S),$ and
$C$ the lexicographic sum over $(S,\pq_0)$ of
the chains $C_i$ $(x_i\in S).$
Thus, $C$ is a chain of $\sum_{x_i\in S}\,m_i$ elements.
To specify a linearization of $P_0*(C_1,\dots,C_{m_0}),$
we can first specify a linearization $Q'$ of $Q,$ then specify how the
elements of the chain $C$ are to be inserted among those of $Q'.$
To complete the proof of the theorem,
it will suffice to show that for each of the finitely
many linearizations $Q'$
of $Q,$ the number of ways of positioning the members of $C$
among the members of $Q'$
is a polynomial in the variables $m_i$ $(x_i\in S).$

Given $Q',$ we shall first show that there is a polynomial
which gives that number
whenever {\em positive} integer values are assigned to the
$m_i$ $(i\in S),$ then show that a polynomial with
this property must continue to give that linearization-count when
its arguments are allowed to be zero.

The complication in counting ways that members of $C$ can
be positioned is that the regions
of $Q'$ where the elements of the various
subchains $C_i$ of $C$ can be inserted
are possibly overlapping intervals, so that information
on where members of one of
those subchains are distributed within its allowed
interval may or may not restrict where
other subchains $C_{i'}$ can be distributed within theirs.

But suppose now that the $m_i$ with $x_i\in S$ are all nonzero,
so that the $C_i$ are all nonempty.
In this case, let us further classify
linearizations of $Q'\cup C$ according to the positions,
relative to the elements of $Q',$ of
the {\em greatest} elements of each of the chains $C_i$ $(x_i\in S);$
i.e., according to which successive
pair of the finitely many elements of $Q'$ each of these
greatest elements
lies between, or whether it lies above or below all of $Q'.$
Once such a set of positions has been specified, we know the linear
ordering of the union of $Q'$ with the set of maximal
elements of those chains (since $S$ is linearly ordered, hence
so is the set of maximal elements of the $C_i$ with $m_i\in S).$
Let us call that union, so ordered, $Q''.$
Then the intervals of $Q''$ in which the {\em remaining}
elements of each $C_i$ can be inserted are disjoint.
If, for a given $C_i,$ that interval of $Q''$
contains $d_i$ elements (where $d_i$
may be $0$ if, for instance, the top members of
$C_i$ and of the next lower $C_{i'}$ have
been placed between the same pair
of elements of $Q'),$ then there are $\binom{m_i-1+d_i}{d_i}$ ways
to populate it with the $m_i-1$ nonmaximal members of $C_i.$
This is a polynomial in $m_i$ (of degree $d_i),$ hence
the total number of linearizations of $P_0*(C_1,\dots,C_{m_0})$
extending our linearization of $Q''$ is
\begin{equation}\begin{minipage}[c]{35pc}\label{d.prod_binom}
$\prod_{x_i\in S} \binom{m_i-1+d_i}{d_i},$
\end{minipage}\end{equation}
which is a polynomial in the variables $m_i$ $(x_i\in S).$
Summing over the finitely many linearizations $Q'$ of $Q,$ and
the finitely many ways of positioning within
each $Q'$ the highest members of the chains $C_i,$
we get a polynomial $f$ which yields
the value of $L_+(P_0;\,m_1,\dots,m_{m_0})$ -- provided that, as
assumed above, none of $m_i$ with $x_i\in S$ is zero.

The key to proving that
the same polynomial works if one or more of those $m_i$ is zero
is Proposition~\ref{P.poly}(a),
which says that if we fix all but one of the $m_i,$
then $L_+(P_0;\,m_1,\dots,m_{m_0})$ is a polynomial in that variable.
Hence, if we choose for all of our $m_i$ other
than some particular $m_{i_0}$ nonzero values, then
the resulting polynomial function in the variable $m_{i_0}$
must agree with the polynomial obtained
in the preceding paragraph at all values of $m_{i_0}$ except
possibly $0.$
But two $\!1\!$-variable real
polynomials which agree at infinitely inputs are equal,
so in fact they will also agree when $m_{i_0}=0.$
This shows that the result of the preceding paragraph extends to the
case where {\em at most one} of the $m_i$ is zero.
An obvious induction on the number of zero arguments gives
the general case.
\end{proof}

The easy part of the above result, saying that if
$S$ is {\em not} a chain, then
$L_+(P_0;\,m_1,\dots,m_{m_0})$ does {\em not} yield polynomial
functions of the corresponding set of variables,
does not have an obvious analog for the sign-imbalance
function $L_-(P_0;\,m_1,\dots,m_{m_0}),$
since it is based on finding a summand that is
too big to be a polynomial; but in the
computation of $L_-,$ ``big'' summands of opposite sign can cancel.
And indeed, we have seen classes of cases where the resulting
function is $0,$ which is certainly a polynomial.

However, the hard part of the above theorem, concerning the case where
$S$ is a chain, does go over to $L_-,$ mutatis mutandis:

\begin{corollary}[to proof of Theorem~\ref{T.chains}] \label{C.chains}
Let $P_0=(|P_0|,\pq_0)$ be as in Theorem~\ref{T.chains},
and let $S\subseteq |P_0|$ be a chain.
Then if we fix nonnegative values for all the $m_i$ with $x_i\notin S,$
the values of $L_-(P_0;\,m_1,\dots,m_{m_0})$ are given
by $2^{\r{card}(S)}$ polynomials in the $m_i$ such that $x_i\in S,$
one for each choice of the parities of these $\r{card}(S)$ variables.
\end{corollary}

\begin{proof}[Idea of proof]
Mimic the argument in the proof of Theorem~\ref{T.chains}
using, in place of the polynomials giving binomial coefficients
as functions of one of their variables,
the pairs of polynomials similarly describing the sign-imbalance of
a union of two chains,
illustrated by the right-hand side of~\eqref{d.Pascal} and
formalized in Proposition~\ref{P.poly}(b).
\end{proof}


\subsection{The total degrees of our multivariable polynomials}\label{SS.deg}
Proposition~\ref{P.poly}(a) described the degree of the one-variable
polynomial it referred to.
We can similarly determine the total degree of the
$\!\r{card}(S)\!$-variable polynomial of Theorem~\ref{T.chains}.
In doing so, we will use the following curious lemma,
which says that one can strengthen the ordering
of a poset so as to get rid of one class of incomparability
conditions, while preserving ``enough of''
another related class of such conditions.

\begin{lemma}\label{L.incomp}
Let $P=(|P|,\pq)$ be a finite poset, and $S$ a chain in $P.$
Then the ordering $\pq$ of $P$ can be strengthened to an
ordering $\pq'$ under which the complement of $S$ in $P$ also becomes
a chain, while every element of $P$ that is
incomparable with at least one element of $S$ under $\pq$ remains
incomparable with at least one element of $S$ under ${\pq'}.$
\end{lemma}

\begin{proof}
Given any pair of incomparable elements $x,y\in|P|-S,$
consider the strengthenings $\pq_{x,y}$ and $\pq_{y,x}$
of $\pq$ obtained by imposing the relation $x\pq_{x,y} y,$
respectively $y\pq_{y,x} x.$
We shall show that
\begin{equation}\begin{minipage}[c]{35pc}\label{d.xy_or_yx}
At least one of $\pq_{x,y},$ $\pq_{y,x}$ has the property
that every element of $P$ that is incomparable under
$\pq$ with at least one element of $S$ remains incomparable
under that strengthened ordering with at least one element of $S.$
\end{minipage}\end{equation}
Repeatedly strengthening our partial order in this way, we eventually
get an ordering under which $|P|-S$ has no incomparable elements,
i.e., is a chain, as desired.

It is easy to check (and probably well-known -- reference, anyone?)
that the relation $\pq_{x,y}$ can be
characterized by the condition that for all $w,z\in|P|,$
\begin{equation}\begin{minipage}[c]{35pc}\label{d.pq_xy}
$w\pq_{x,y} z$ if and only if either $w\pq z,$
or $w\pq x$ and $y\pq z.$
\end{minipage}\end{equation}
Thus, if an element $u\in|P|$ satisfies a relation
under $\pq_{x,y}$ that it does {\em not} satisfy under $\pq,$
we must have either $u\pq x$ or $y\pq u.$
(Note, incidentally, that these two cases are mutually exclusive, since
$x$ and $y$ were assumed $\pq\!$-incomparable.)

Suppose now that we have an element $u\in|P|$ such that
\begin{equation}\begin{minipage}[c]{35pc}\label{d.u}
$u$ is $\!\pq\!$-incomparable with at least one element of $S,$ but
is $\!\pq_{x,y}\!$-comparable with all elements of $S.$
\end{minipage}\end{equation}
Of the two alternatives noted at the end of the sentence
following~\eqref{d.pq_xy}, let us begin by assuming
\begin{equation}\begin{minipage}[c]{35pc}\label{d.u_pq_x}
$u\,\pq\,x.$
\end{minipage}\end{equation}
Let us write ${\downarrow}u$ for $\{z\in|P|\mid z\pq u\}$
and ${\downarrow_{x,y}}\,u$ for $\{z\in|P|\mid z\pq_{x,y} u\},$
and make the obvious corresponding definitions of
${\uparrow}u$ and ${\uparrow_{x,y}}\,u.$
The statement that $u$ is $\!\pq_{x,y}\!$-comparable with all
elements of $S$ thus says that
\begin{equation}\begin{minipage}[c]{35pc}\label{d.cup=S}
$({\downarrow_{x,y}}\,u\cap S)\,\cup\,({\uparrow_{x,y}}\,u\cap S)
\ =\ S.$
\end{minipage}\end{equation}
Now from~\eqref{d.pq_xy} and~\eqref{d.u_pq_x}, it follows that
\begin{equation}\begin{minipage}[c]{35pc}\label{d.down+up}
${\downarrow_{x,y}}\,u ={\downarrow}u$ \ and
\ ${\uparrow_{x,y}}\,u\ =\ {\uparrow}u \cup {\uparrow}y.$
\end{minipage}\end{equation}
(The first relation is gotten by putting $u$ in the role of $z$
in~\eqref{d.pq_xy}, and letting
$w$ range over $P;$ the second by putting
$u$ in the role of $w$ and letting $z$ range over $P.)$
Since $S$ is a chain, $({\uparrow}u \cup {\uparrow}y)\cap S$
must be either ${\uparrow}u\cap S$ or ${\uparrow}y\cap S.$
If it were ${\uparrow}u\cap S,$ then in view
of~\eqref{d.down+up},~\eqref{d.cup=S} would
say that $u$ was $\!\pq\!$-comparable with all
elements of $S,$ contrary to the first condition of~\eqref{d.u}.
Hence it is ${\uparrow}y\cap S,$ and~\eqref{d.cup=S} instead says
\begin{equation}\begin{minipage}[c]{35pc}\label{d.cup=S.2}
$({\downarrow}u\cap S)\,\cup\,({\uparrow}y\cap S)\ =\ S.$
\end{minipage}\end{equation}
In view of~\eqref{d.u_pq_x}, ${\downarrow}u\subseteq {\downarrow}x,$
so~\eqref{d.cup=S.2} implies
\begin{equation}\begin{minipage}[c]{35pc}\label{d.cup=S.3}
$({\downarrow}x\cap S)\,\cup\,({\uparrow}y\cap S)\ =\ S.$
\end{minipage}\end{equation}

If, rather than~\eqref{d.u_pq_x} we are in the other case,
$y\pq u,$ we get the variant of~\eqref{d.cup=S.3} with $\uparrow$
and $\downarrow$ reversed
and the roles of $x$ and $y$ interchanged -- which
is again~\eqref{d.cup=S.3}.
So~\eqref{d.cup=S.3} holds in either case.

Now suppose that in addition to an element $u$ satisfying~\eqref{d.u},
there is also an element $v\in|P|$ such that
\begin{equation}\begin{minipage}[c]{35pc}\label{d.v}
$v$ is $\!\pq\!$-incomparable with at least one element of $S,$
but is $\!\pq_{y,x}\!$-comparable with all elements of~$S.$
\end{minipage}\end{equation}
Then we get the variant of~\eqref{d.cup=S.3} with only
the roles of $x$ and $y$ interchanged:
\begin{equation}\begin{minipage}[c]{35pc}\label{d.cup=S.4}
$({\downarrow}y\cap S)\,\cup\,({\uparrow}x\cap S)\ =\ S.$
\end{minipage}\end{equation}

From~\eqref{d.cup=S.3} and~\eqref{d.cup=S.4}, it is not hard
to deduce that ${\downarrow}x\cap S ={\downarrow}y\cap S$
and ${\uparrow}x\cap S ={\uparrow}y\cap S,$ and that these
are complementary subsets of $S;$ i.e., that both $x$
and $y$ are $\!\pq\!$-comparable with all elements of $S,$ and
that the order relation of each element of $S$ with $x$
is the same as its order relation with $y.$
(Quick Venn diagram proof: Draw a square, to represent
properties of an element $s\in S.$
Divide it by vertical lines into three regions according to the
$\!\pq\!$-relation of $s$ to $x:$ ``smaller'', ``incomparable'',
or ``greater''; and similarly by horizontal lines according
to its $\!\pq\!$-relation to $y.$
Interpret each of~\eqref{d.cup=S.3} and~\eqref{d.cup=S.4} as saying
that every $s\in S$ lies in a certain region of this diagram.
Shade those regions, and note their intersection.)

It follows from this and~\eqref{d.pq_xy}
that passing from $\pq$ to $\pq_{x,y}$
does not affect the order-relation or lack of it between any element
of $P$ and any element of $S$ (and similarly for $\pq_{y,x}).$
This contradicts our assumption~\eqref{d.u} (and also~\eqref{d.v}).
Thus, our assumption that there existed both $u$ satisfying~\eqref{d.u}
and $v$ satisfying~\eqref{d.v} has led to a contradiction,
proving~\eqref{d.xy_or_yx}, and completing the proof of the lemma.
\end{proof}

Some further observations related to the above lemma
are noted in an appendix, section~\ref{S.re_L.incomp}.

We can now get the result we are aiming for.
Recall that the {\em total degree} of a polynomial $f$ in
several variables means
the maximum, over the monomials occurring in $f,$ of the sum
of the exponents of the variables.

\begin{corollary}[to proof of Theorem~\ref{T.chains}]\label{C.degree}
Let $P_0$ be a finite poset and $S$ a chain in $P_0,$
and, as in Theorem~\ref{T.chains}, consider the function
$L_+(P_0;\,m_1,\dots,m_{m_0})$ with fixed values chosen
for all the $m_i$ with $x_i\notin S,$ as a function of
the values of $m_i$ with $x_i\in S.$
Let $f$ be the polynomial in the latter $\r{card}(S)$
variables which, by that theorem, gives this function.

Then the total degree of $f$ is equal to the sum of $m_i$ over those
$x_i$ that are incomparable in $P_0$ with at least one element of $S.$
\end{corollary}

\begin{proof}[Sketch of proof]
The degree we are looking for will be the
maximum of the total degrees, in the variables
corresponding to the elements of $S,$
of the {\em leading terms} of the polynomials~\eqref{d.prod_binom}
that are summed to get $f.$
(A general multivariable polynomial does not have a
well-defined ``leading term'', but the meaning
of that phrase for the polynomials~\eqref{d.prod_binom} is clear.
When we sum the terms~\eqref{d.prod_binom},
the leading terms of some of these polynomials may
be cancelled by negative-coefficient non-leading terms of
higher-degree polynomials;
but this cannot happen to any of the leading terms
that have maximum total degree.)
Now for each such polynomial, its total degree in the variables
we are interested in is the sum over $x_i\in S$ of the
number called $d_i$ in the proof of Theorem~\ref{T.chains}; i.e., the
number of elements of $Q$ lying
in the range into which the non-maximal
members of the subchain $C_i$ can be placed.
Every element in one of those ranges must belong to a $C_i$
$(x_i\in |P|-S)$
such that $x_i$ is incomparable with at least one member
of $S;$ so the sum of the cardinalities of those chains
$C_i,$ i.e., the sum of the corresponding $m_i,$
is indeed an upper bound for the desired degree.

To get a linearization $Q'$ of the $Q$ of that proof with
the help of which we can realize
that upper bound, we apply Lemma~\ref{L.incomp},
taking for the $P$ of that lemma the lexicographic
sum over $P_0$ having a chain of length $m_i$ in the $\!i\!$-th
position for all $i$ with $x_i\notin S,$ while for $x_i\in S,$
it has a singleton, which we continue to denote $x_i.$
Taking the set of these singletons for the $S$ of that lemma,
we get an ordering $\pq'$ of $P$ that makes $|P|-S = |Q|$ a chain,
and we take this linearization of $Q$ to be our $Q'.$

To single out a polynomial~\eqref{d.prod_binom}, we now need to choose
for each $i$ with $x_i\in S$ the position where, in the
construction of Theorem~\ref{T.chains}, we will locate
the largest element of the chain
$C_i$ relative to the elements of $Q'.$
Let us place each of these as high as we can consistent with
the order $\pq'.$
I.e., for the top $x_i\in S,$ we place
$x_{i,m_i}$ just below the least element of $Q'$ that,
under the ordering constructed, is above $x_i,$
if there is one; if not, we place it
above all elements of $Q';$ and we then place the top elements
of successively lower chains $C_i$ for $x_i\in S$ as high as they
can go relative to $Q'$ and the elements we have put down so far.
One finds that every element of $Q'$ incomparable
under $\pq'$ with at least one member of $S$ is in the range which,
in the construction of Theorem~\ref{T.chains}, can be populated
by elements of $C_i$ for some $x_i\in S;$ so the
asserted total degree is achieved.
\end{proof}

\section{Associativity of the lexicographic sum, and its consequences}\label{S.iteration}
Suppose, as in our general description of lexicographic sums, that
$P_0=(|P_0|,\pq_0)$ is a poset with $|P_0|=\{x_1,\dots,x_{m_0}\},$
and that for each $i\in\{1,\dots,m_0\}$ we are given a poset
$P_i=(|P_i|,\pq_i)$ with $|P_i|=\{x_{i,1},\dots,x_{i,m_i}\},$
where $x_{i,j}$ and $x_{i',j'}$ are distinct unless $(i,j)=(i',j').$

Now suppose further that for each pair $(i,j)$
with $1\leq i\leq m_0$ and $1\leq j\leq m_i,$ we are
given a poset $P_{i,j}=(|P_{i,j}|,\pq_{i,j})$ with
$|P_{i,j}|=\{x_{i,j,1},\dots,x_{i,j,m_{i,j}}\},$
such that $x_{i,j,k}$ and $x_{i',j',k'}$ are distinct
unless $(i,j,k)=(i',j',k').$
Then we can define a partial ordering on
\begin{equation}\begin{minipage}[c]{35pc}\label{d.xijk}
$|P|\ =\ \{x_{i,j,k}\mid 1\leq i\leq m_0,
\ 1\leq j\leq m_i,\ 1\leq k\leq m_{i,j}\}$
\end{minipage}\end{equation}
by letting
\begin{equation}\begin{minipage}[c]{35pc}\label{d.vijk}
$x_{i,j,k}\pq x_{i',j',k'}$ if and only if either\\[.1em]
\hspace*{1em}$i\neq i'$ and $x_i\pq_0 x_{i'}$
in $P_0,$ or\\
\hspace*{1em}$i=i'$ but $j\neq j',$ and $x_{i,j}\pq_i x_{i,j'}$
in $P_i,$ or\\
\hspace*{1em}$i=i'$ and $j=j',$ and $x_{i,j,k}\pq_{i,j} x_{i,j,k'}$
in $P_{i,j}.$
\end{minipage}\end{equation}
Clearly, the resulting poset $P$ can
be looked at both as the lexicographic sum over $P_0$
of the posets $P_i*(P_{i,1},\dots,P_{i,m_i})$ and
as the lexicographic sum over $P_0*(P_1,\dots,P_{m_0})$ of
the posets $P_{i,j}.$
Denoting their common value $P_0*(P_i)_{1\leq i\leq m_0}*
(P_{i,j})_{1\leq i\leq m_0,\,1\leq j\leq m_i},$ we thus have
\begin{equation}\begin{minipage}[c]{35pc}\label{d.P*P*P}
$P_0*((P_i)*(P_{i,j})_{1\leq j\leq m_i})_{1\leq i\leq m_0}\\[.2em]
\hspace*{2em}=\ P_0*(P_i)_{1\leq i\leq m_0}*
(P_{i,j})_{1\leq i\leq m_0,\,1\leq j\leq m_i}\\[.2em]
\hspace*{4em}=\ (P_0*(P_i)_{1\leq i\leq m_0})*
(P_{i,j})_{1\leq i\leq m_0,\,1\leq j\leq m_i}.$
\end{minipage}\end{equation}
The equality between the first and last lines of~\eqref{d.P*P*P}
constitutes an associative law for lexicographic sums.
(Of course, if one is given families of posets essentially as above,
but without the disjointness assumptions on their underlying sets,
one can construct
lexicographic sums using ordered tuples, as at the beginning
of section~\ref{SS.lex}, and one gets natural {\em isomorphisms},
rather than equalities, in~\eqref{d.P*P*P}.)

Now let us suppose each of the posets $P_{i,j}$ is a chain
$C_{i,j}$ of $m_{i,j}$ elements, but make no such
assumption on $P_0$ or the $P_i.$
Then if we apply the function $L_\pm$ to the expressions
in~\eqref{d.P*P*P}, the final expression is the function
\mbox{$L_\pm(P_0*(P_i)_{1\leq i\leq m_0};
m_{1,1},\dots,m_{m_0,m_{m_0}}),$}
while the initial expression can be computed using Theorem~\ref{T.L(*)}
from the functions $L_\pm(P_0;\,m_1,\dots,m_{m_0})$ and
$L_\pm(P_i;\,m_{i,1},\dots,m_{i,m_i})$ $(i=1,\dots,m_0).$

Thus, if we know the functions $L_\pm(P;\,m_1,\dots)$
for some family of finite posets $P,$ we can compute
using Theorem~\ref{T.L(*)} the corresponding
function for any poset constructed as a lexicographic sum
of members of that family over a member of that family; and,
more generally, for any poset
obtained in that way by iterated lexicographic sums.

In particular, since Propositions~\ref{P.chain} and~\ref{P.antich} give
formulas for $L_\pm(P;\,m_1,\dots)$
when $P$ is a finite chain or antichain, we can use the
above technique to get such formulas for all posets constructed
from chains and antichains by iterated lexicographic sums.

For example, consider the poset $P=$
\raisebox{0.5pt}[7pt][5pt]{ 
\begin{picture}(15,12)
\dotline{0,-3}{0,8}
\dotline{0,8}{10,-3}
\dotline{10,-3}{10,8}
\dotline{10,8}{0,-3}
\end{picture}}.
If we name the bottom two elements $x_{1,1}$ and $x_{1,2},$
and the top two $x_{2,1}$ and $x_{2,2},$ then our poset
is the lexicographic sum of two antichains with underlying sets
$\{x_{1,1}, x_{1,2}\}$ and $\{x_{2,1}, x_{2,2}\},$
over a chain whose element-set we may label $\{x_1,x_2\}.$
With the help of Propositions~\ref{P.chain} and~\ref{P.antich},
we find that $L_+(P;\,m_{1,1},m_{1,2},m_{2,1},m_{2,2})=
\binom{\,m_{1,1}+m_{1,2}}{\ m_{1,1}}
\binom{\,m_{2,1}+m_{2,2}}{\ m_{2,1}},$
while $L_-(P;\,m_{1,1},m_{1,2},m_{2,1},m_{2,2})$
is zero if either both $m_{1,1}$ and $m_{1,2}$ are odd, or
both $m_{2,1}$ and $m_{2,2}$ are odd, while it is
$\binom{\lfloor(m_{1,1}+m_{1,2})/2\rfloor}{\lfloor m_{1,1}/2\rfloor}
\binom{\lfloor(m_{2,1}+m_{2,2})/2\rfloor}{\lfloor m_{2,1}/2\rfloor}$
otherwise.
Using iterated lexicographic sums, one can build up
from chains and antichains arbitrarily complicated posets
for which these functions can similarly be computed.
These are called ``series-parallel'' posets
in \cite[Chapter~9, Exercise~6]{BSWS}.

But ``most'' finite posets are not series-parallel;
the simplest example is $P=\!$
\raisebox{0.5pt}[7pt][5pt]{ 
\begin{picture}(15,12)
\dotline{0,-3}{5,8}
\dotline{5,8}{10,-3}
\dotline{10,-3}{15,8}
\put(15,8){\circle*{2}}
\end{picture}}.
In fact, it is shown in \cite[Chapter~9, Exercises~6--7]{BSWS}
that a finite poset is series-parallel if and only
if it does not contain a copy of that $\!4\!$-element poset.
For further results on the characterization of classes of posets
arising as iterated lexicographic sums
in terms of ``forbidden subposets'', see~\cite{forbidden}.

As mentioned at the end of section~\ref{S.lex},
I have not studied the function $L_\pm(P;\,m_1,m_2,m_3,m_4)$
determined by the above \mbox{$\!4\!$-element} poset
\raisebox{0.5pt}[7pt][5pt]{ 
\begin{picture}(15,12)
\dotline{0,-3}{5,8}
\dotline{5,8}{10,-3}
\dotline{10,-3}{15,8}
\put(15,8){\circle*{2}}
\end{picture}}.
It would be interesting to describe it.

\section{Appendix: notes on Lemma~\ref{L.incomp}}\label{S.re_L.incomp}

I don't know whether Lemma~\ref{L.incomp} has uses
other than as a tool for proving Corollary~\ref{C.degree}; but it
has piqued my curiosity, and I give below several related observations.

First, some quick examples.
For a case of~\eqref{d.xy_or_yx} in which one, but not the other
of $\pq_{x,y}$ and $\pq_{y,x}$ has the property asserted there,
let $P$ be the $\!4\!$-element poset
\raisebox{0.5pt}[7pt][5pt]{ 
\begin{picture}(15,12)
\dotline{0,-3}{5,8}
\dotline{5,8}{10,-3}
\dotline{10,-3}{15,8}
\put(15,8){\circle*{2}}
\end{picture}},
let $S$ be the $\!2\!$-element chain in the middle of that picture,
and let $x$ and $y$ be (necessarily) the two elements of $|P|-S.$
The reader can easily check the details.

For a case of the lemma in which the set of
elements of $S$ incomparable with one or more elements
of $P$ must always decrease when $|P|-S$ is made a chain
(in contrast to the assertion of the lemma,
about elements of $P$ incomparable with one or more elements of $S),$
let $P$ be the disconnected union of a $\!3\!$-element chain
and a singleton, and take for $S$ the subchain
consisting of the top and bottom elements
of the $\!3\!$-element component.

Finally, for an example showing that the statement of the lemma fails
if we drop the assumption that $S$ is a chain,
let $P$ be the disconnected union of two $\!2\!$-element chains, and
let $S$ consist of the two minimal elements of $P.$
Then under any strengthening of the ordering of $P$ that
makes $|P|-S$ a chain, the larger element of $|P|-S$
will lie above, and in
particular, be comparable with, both elements of $S,$
though it was incomparable with one of them under the original ordering.

It is not obvious from the proof we gave of Lemma~\ref{L.incomp}
how to tell, given elements $x$ and $y,$ which
of the orderings $\pq_{x,y}$ and $\pq_{y,x}$ has the property
asserted in~\eqref{d.xy_or_yx}, or whether both do.
The equivalence (a)$\!\iff\!$(c) of the
next corollary, and its variant with the roles of
$x$ and $y$ reversed, give the criteria for
one or the other of those possibilities to be excluded.

\begin{corollary}[to proof of Lemma~\ref{L.incomp}]\label{S.xy_or_yx}
Let $P=(|P|,\pq)$ be a finite poset and $S$ a chain in $P,$
let $x$ and $y$ be two $\!\pq\!$-incomparable elements of $|P|-S,$
and let $\pq_{x,y}$ be the strengthened partial ordering on $P$
gotten by imposing the relation $x\pq_{x,y} y,$
described in~\eqref{d.pq_xy}.
Then the following three conditions are equivalent.

\textup{(a)} \ There is some $u\in |P|$ that is incomparable under
$\pq$ with at least one element of $S,$ but
is comparable under $\pq_{x,y}$ with every element of $S.$

\textup{(b)} \ One of $x$ or $y$ has the above property of being
incomparable under $\pq$ with at least one element of $S,$ but
comparable under $\pq_{x,y}$ with every element of $S.$\vspace{.1em}

\textup{(c)} \ $({\downarrow}x\cap S)\cup({\uparrow}y\cap S)\,=
\,S,$ but $({\downarrow}y\cap S)\cup({\uparrow}x\cap S)
\,\neq\,S.$
\textup{(}Cf.~\eqref{d.cup=S.3},~\eqref{d.cup=S.4}\textup{)}.
\end{corollary}

\begin{proof}
We shall prove (a)$\implies$(c)$\implies$(b)$\implies$(a).

Given~(a), the proof of Lemma~\ref{L.incomp}
gives us the equality~\eqref{d.cup=S.3} with which~(c) begins.
On the other hand, if alongside~\eqref{d.cup=S.3}
we have equality in place of the inequality in the second condition
of~(c) (i.e., if ~\eqref{d.cup=S.4} also holds),
then the two paragraphs of that proof
following~\eqref{d.cup=S.4} show that $\pq_{x,y}$ does not
introduce any relations between elements of $|P|-S$ and elements of $S$
that don't hold under $\pq,$ contradicting~(a).
So we must also have that inequality; so~(c) indeed holds.

Next, assume~(c).
Note that in the equality
$({\downarrow}x\cap S)\cup({\uparrow}y\cap S) = S,$
the sets ${\downarrow}x\cap S$ and ${\uparrow}y\cap S$
must be disjoint, since if their intersection contained an element $z,$
we would have $y\pq z\pq x,$ contradicting our assumption that
$x$ and $y$ are $\!\pq\!$-incomparable.

Let us use the inequality
$({\downarrow}y\cap S)\cup({\uparrow}x\cap S)\neq S$
to choose a $w\in S$ in neither
${\downarrow}y$ nor ${\uparrow}x.$
By the preceding observation, $w$ lies
either in ${\downarrow}x$ or in ${\uparrow}y,$ but not in both.
If it is in ${\downarrow}x$ but not ${\uparrow}y,$
then it is in neither
${\uparrow}y$ nor ${\downarrow}y,$ showing that $y$ is incomparable
with at least one element of $S.$
Now every $s\in S$ incomparable with $y$ must, by the relation
$({\downarrow}x\cap S)\cup({\uparrow}y\cap S) = S,$
lie in ${\downarrow}x,$ hence $s\pq_{x,y} x\pq_{x,y} y.$
So $y,$ though $\!\pq\!$-incomparable with some elements $s\in S,$
is $\!\pq_{x,y}\!$-comparable with all such elements, which is
the ``$\!y\!$'' case of~(b).
If, on the other hand, our element $w$ lies
in ${\uparrow}y$ rather than ${\downarrow}x,$
the corresponding considerations give the ``$\!x\!$'' case of~(b).
Thus, we have proved (c)$\!\implies\!$(b).

The implication (b)$\!\implies\!$(a) is immediate.
\end{proof}

The proof of Lemma~\ref{L.incomp}
shows how to build up {\em all} linearizations of $|P|-S$ which
extend to orderings of $|P|$ of the sort asserted in the
lemma, by making successive choices of
order on unordered pairs of elements of $|P|-S,$ taken in any order.
The above corollary tells us at each such
step which choices are available
(those {\em not} satisfying~(c)).
We end this section with a more systematic construction of {\em some}
ordering as in that lemma, based on a suggestion of Stefan Felsner
(personal correspondence).

\begin{proof}[Sketch of an alternative proof of Lemma~\ref{L.incomp},
after S.\,Felsner]
Listing the elements of $S$ as $s_1\pq\dots\pq s_r,$ let us
partition $|P|-S$ into disjoint subsets
\begin{equation}\begin{minipage}[c]{35pc}\label{d.T_1...}
$|P|-S\ =\ T_1\ \cup\ T_2\ \cup\ \dots\ \cup\ T_{2r+1}$
\end{minipage}\end{equation}
as follows.
If $x\in|P|-S$ is already
$\!\pq\!$-comparable with all elements of $S,$ say
with $s_i\pq x\pq s_{i+1},$ we assign $x$ to $T_{2i+1},$ with
the obvious modifications in the end-cases, namely, when $x\pq s_1$
we assign it to $T_1,$ and when $s_r\pq x$ we assign $x$ to $T_{2r+1}.$
On the other hand, if $x$ is incomparable with at least one element
of $S,$ let $s_i$ be the largest such element,
and assign $x$ to $T_{2i}.$

It is not hard to check that for $1\leq i< j\leq 2r+1,$
no element of $T_j$ is $\pq$ any element of $T_i.$
Hence we can strengthen the ordering $\pq$ on $|P|-S$ to
make all elements of $T_i$ precede all elements of $T_j$
whenever $i<j,$ keeping the relative order of elements
within each $T_i.$
We can then go further and linearize each $T_i,$
getting a total order $\pq'$ on $|P|-S.$

On $S,$ on the other hand, we let $\pq'$ agree
with $\pq,$ since $\pq$ is already a total order there.

It remains to specify how $\pq'$ should relate elements
of $|P|-S$ and elements of $S.$
If $x\in|P|-S$ belongs to a set $T_{2i+1},$
there is no choice: under $\pq,$ $x$ lies above
all $s_j$ with $j\leq i$ and below all $s_j$ with $j\geq i+1,$
so we give it these same relations under $\pq'.$

If $x\in T_{2i},$ we must again let $x$ lie
below all $s_j$ with $j\geq i+1.$
In this case, there may or may not be choices as to how it
should relate to lower members of $S;$ but we make a choice that
will always work: let $x$ be incomparable with $s_i,$
and lie above all $s_j$ with $j< i.$

It is routine, though tedious, to verify that the relation
$\pq'$ so defined is a partial ordering on $|P|.$
By construction, it is a strengthening of the given ordering $\pq,$
and has the property that
every element of $|P|-S$ that was $\!\pq\!$-incomparable with at
least one element of $S$ (i.e., which belongs to some $|T_{2i}|),$
remains $\!\pq'\!$-incomparable with some element of
$S$ (namely, $s_i).$
This completes the proof of the lemma.
\end{proof}

\section{Appendix: a formula of G.\,Hochschild}\label{S.GPH}

The results of section~\ref{S.ord+} were motivated by a question
Arthur Ogus asked me, on how one might understand,
computationally, a formula of Gerhard Hochschild.
In this appendix, which assumes only that section,
we recover that formula.

Our development is far lengthier than Hochschild's, so its interest
(if any) lies in its different approach to the result,
and in the possibility that the method may be applicable to
questions not as easy to answer by other means.

Hochschild's result (in which I have changed almost
all the notation -- but the translation between his
and mine is straightforward) concerns an
associative ring $R$ of prime characteristic $p,$
a commutative subring $A$ of $R,$ and an element $r\in R$ such
that the commutator map $a\mapsto \ad{r}(a)=ra-ar$
carries $A$ into itself.
What he shows is that for all $a\in A,$
\begin{equation}\begin{minipage}[c]{35pc}\label{d.GPH}
$(ar)^p\ =\ a^p r^p\,+\,\ad{ar}^{p-1}(a)\,r.$
\end{minipage}\end{equation}

Since every term of~\eqref{d.GPH} ends with a factor $r,$
Ogus suggested that the corresponding identity with those
factors removed should hold, namely
\begin{equation}\begin{minipage}[c]{35pc}\label{d.AO}
$(ar)^{p-1}a\ =\ a^p r^{p-1} +\,\ad{ar}^{p-1}(a).$
\end{minipage}\end{equation}

We shall see that this is true.
Precisely, dropping the assumption that $R$ has characteristic
$p,$ we shall prove

\begin{lemma}[{after Hochschild \cite[Lemma~1]{GPH}}]\label{L.GPH}
Let $R$ be an associative ring, $r$ an element
of $R,$ $A$ a commutative subring
of $R$ such that the operation $\ad{r}:a\mapsto ra-ar$
carries $A$ into itself, and $p$ a prime number.
Then the function taking every $a\in A$ to the element
\begin{equation}\begin{minipage}[c]{35pc}\label{d.AO-}
$(ar)^{p-1}a\ -\ a^p r^{p-1}\ -\ \ad{ar}^{p-1}(a)$
\end{minipage}\end{equation}
of $R$
can be written as a noncommutative polynomial in $r$ and the elements
$a,$ $\ad{r}(a),$ $\ad{r}^2(a),$ \dots\,, in which the coefficients
of all monomials are divisible by~$p.$
\end{lemma}

\begin{proof}
Let us consider the term $(ar)^{p-1}a$ of~\eqref{d.AO-}, and
repeatedly use the formula
\begin{equation}\begin{minipage}[c]{35pc}\label{d.rx}
$rx\ =\ xr\,+\,\ad{r}(x)$\quad for $x\in A$
\end{minipage}\end{equation}
to eliminate occurrences of $r$ preceding elements of $A.$
We begin with the rightmost occurrence of $r,$ which
precedes the final $a;$ an application of~\eqref{d.rx} to that pair of
factors turns $(ar)^{p-1}a$ into a sum of two monomials,
in one of which that $r$ has jumped to the end, while in the
other, it has been absorbed in the process of turning the final $a$
to $\ad{r}(a).$
We then apply~\eqref{d.rx} to the $r$ that was originally
second from the right.
This can either be absorbed in the $a$ immediately to its
right, turning that into $\ad{r}(a),$ or jump past it.
In the latter case it can, in turn, either be absorbed in the
next factor (which is $a$ or $\ad{r}(a)$ depending on which output
of the first step we are looking at, and so is turned into
$\ad{r}(a)$ or $\ad{r}^2(a)$ respectively), or jump past that
factor, becoming an (additional) $r$ at the far right.
We proceed similarly with the third $r$ from the right, and so on.
Since there are two possibilities for the fate of the rightmost $r,$
three for the next, etc., we get $p\,!$ terms.

The choices leading to one of these $p\,!$ terms can be represented
visually by taking the given string
\begin{equation}\begin{minipage}[c]{35pc}\label{d.ar...ra}
$a\,r\,a\,r\,\dots\,r\,a$\qquad(with $p$ $\!a\!$'s
and $p-1$ $\!r\!$'s)
\end{minipage}\end{equation}
and drawing, from each member of some subset of the occurrences of $r,$
an arrow to some occurrence of $a$ to its right.
In the resulting expression, each $a$
that receives $m$ arrows becomes $\ad{r}^m(a)$ (with those receiving
none remaining as $a),$ and the $\!r\!$'s at
the other ends of those arrows are deleted,
while those $\!r\!$'s from which
arrows were not drawn move to the far right.

Note that the output of this process has exactly one term in
which none of the $p-1$ $\!r\!$'s acts on any of the $\!a\!$'s, and
that this term, $a^p r^{p-1},$
is cancelled by the $-a^p r^{p-1}$ of~\eqref{d.AO-}.

At the opposite extreme are the $(p-1)!$ terms in which
{\em every} occurrence of $r$ acts on an element of $A$ (either
an $a,$ or the image of $a$ under previous actions
of other occurrences of $r).$
I claim that the sum of these is precisely
$\ad{ar}^{p-1}(a),$ and thus cancels the
$-\,\ad{ar}^{p-1}(a)$ of~\eqref{d.AO-}.
To see this, note first that for any $x\in A,$
we have $\ad{ar}(x) = arx-xar = arx-axr = a\,\ad{r}(x),$
where the middle equality holds because $A$ is commutative.
Hence $\ad{ar}^{p-1}(a)$ can be written as
\begin{equation}\begin{minipage}[c]{35pc}\label{d.a_ad()}
$(a\,\ad{r}(\,\dots\,(a\,\ad{r}(a\,\ad{r}(a)))...))\,.$
\end{minipage}\end{equation}
Here the rightmost (better: innermost) occurrence
of $\ad{r}$ acts on its argument $a.$
The next occurrence acts on the product $a\ \ad{r}(a),$
so -- since $\ad{r}$ is a derivation, i.e., satisfies
\begin{equation}\begin{minipage}[c]{35pc}\label{d.adrxy}
$\ad{r}(x y)\ =\ \ad{r}(x)\,y\,+\,x\,\ad{r}(y)$ \quad for $x,y\in A,$
\end{minipage}\end{equation}
-- it turns $a\ \ad{r}(a),$ into a sum of two terms, in one of which
it acts on the first factor and in the other on the second.
The action of the next $\ad{r},$ on the product of $a$ with
each of these two-factor terms, gives a sum of three terms; and so on.
The resulting terms can be classified by writing
out the string
\begin{equation}\begin{minipage}[c]{35pc}\label{d.a_ad}
$a\ \ad{r}\ \dots\ a\ \ad{r}\ a\ \ad{r}\ a$
\quad (with $p$ $\!a\!$'s and $p-1$ $\!\ad{r}\!$'s)
\end{minipage}\end{equation}
and drawing an arrow from each $\ad{r}$ to an arbitrary
$a$ to the right of it, on which it acts.
The results are clearly the same as the subset of the expressions
in our expansion of $(ar)^{p-1}a$ in which an
arrow comes out of every occurrence of $r,$ so, as claimed,
these terms cancel the $-\,\ad{ar}^{p-1}(a)$ in~\eqref{d.AO-}.

What remains is to show that in the
expansion of $(ar)^{p-1}a,$ each monomial
\begin{equation}\begin{minipage}[c]{35pc}\label{d.monom}
$\ad{r}^{m_1}(a)\ \ad{r}^{m_2}(a)\ \ad{r}^{m_p}(a)
\ r^{p-1-m_1-\ldots-m_p}$
\end{minipage}\end{equation}
that is {\em not} of either of the above two extreme sorts,
i.e., which satisfies
\begin{equation}\begin{minipage}[c]{35pc}\label{d.neq0,p-1}
$0\ <\ m_1+\dots+m_p\ <\ p-1,$
\end{minipage}\end{equation}
occurs with coefficient divisible by $p.$
Since $A$ is commutative, we are regarding monomials~\eqref{d.monom}
as the same if they differ by a permutation of
the string of exponents $m_1,\dots,m_p.$
Thus, we may make the notational assumption
that the {\em nonzero} exponents in~\eqref{d.monom} form
an initial substring, say $m_1,\dots,m_\ell.$

To determine the coefficient of a monomial~\eqref{d.monom},
we need to count the ways of attaching arrows
to~\eqref{d.ar...ra} that lead to it.
I claim such a system of arrows
will be determined by an appropriately indexed family of
elements of $\{1,\dots,p\}$ (corresponding to the positions of the
$\!r\!$'s from which arrows begin, and the $\!a\!$'s at
which they end) subject to certain inequalities --
and that the set of these indexed families can be identified with a
set of maps of the sort whose cardinality was studied
in section~\ref{S.ord+}.

Indeed, consider any monomial~\eqref{d.monom}, and
let $P$ be a partially ordered set consisting
of $\ell$ connected components, namely, for each $i\leq\ell,$
let the component $P_i$ be a chain of $m_i+1$ elements.

Let us, to begin with, assume for simplicity that the nonzero exponents
in~\eqref{d.monom}, $m_1,\dots,m_\ell,$ are distinct.
Then for $P$ so defined, let us, to each diagram of arrows
on the string of symbols~\eqref{d.ar...ra} that yields the
monomial~\eqref{d.monom}, associate the map
$\varphi:P\to\{1,\dots,p\}$ such that for $i=1,\dots,\ell,$
if the factor $\ad{r}^{m_i}(a)$ in~\eqref{d.monom} arises from
arrows drawn from
the $\!k_1\!$-st, $\!k_2\!$-nd, through $\!k_{m_i}\!$-th
occurrences of $r$ to the $\!k_{m_i+1}\!$-st occurrence
of $a,$ then $\varphi$ maps the successive terms of the chain $P_i$
to the integers $k_1<k_2<\dots<k_{m_i}<k_{m_i+1}.$
(Note the difference between the way $k_1,\dots,k_{m_i}$ are specified
in the above sentence, and way $k_{m_i+1}$ is specified.)

Which maps $P\to\{1,\dots,p\}$ can arise in this way from
arrow-diagrams giving the monomial~\eqref{d.monom}?
It is not hard to see that they will be precisely those
{\em isotone} maps such that no two elements of $P$ fall together,
except that the {\em maximal} element of a component $P_i$ is permitted
to fall together with a {\em nonmaximal} element of a component
$P_j$ if $j\neq i.$
(The images of non-maximal elements of $P$ all
have to be distinct because they represent the
sources of distinct arrows, and the images of maximal
elements must be distinct because they represent the
recipients of distinct families of arrows.
Finally, the images of a non-maximal element and the maximal element
in the same component must be distinct because the $\!i\!$-th
occurrence of $r$ can't have an arrow to the $\!i\!$-th
occurrence of $a,$ since the latter precedes it.
On the other hand, there is no contradiction if for some $i,$
the $\!i\!$-th $a$ is
the recipient of some family of arrows, and the $\!i\!$-th $r$
is the source of an arrow with a different destination.
These are the cases where elements
of $P$ are allowed to fall together.)
So letting $S=(|P|,\pq,E),$ where $(|P|,\pq)$ is the
poset $P$ described above, and $E$ consists of all two-element subsets
of $|P|$ other than those whose members are the maximal element
of one component of $P$ and a nonmaximal element of another
component, we see that the coefficient of~\eqref{d.monom}
in $(ar)^{p-1}a$ is $C(S,p),$ as defined in Theorem~\ref{T.ord+}.

The partially ordered set $P$ has
$(m_1+1)+\dots+(m_\ell+1)=m_1+\dots+m_\ell+\ell$
elements, and $\ell$ connected components, so the term written
$\r{card}(|S|)-c+1$ in Corollary~\ref{C.==0} is
here $m_1+\dots+m_\ell+1.$
By~\eqref{d.neq0,p-1}, $m_1+\dots+m_\ell<p-1,$
hence $m_1+\dots+m_\ell+1<p,$
so $p$ satisfies the condition in the first
sentence of Corollary~\ref{C.==0}.
By the first inequality of~\eqref{d.neq0,p-1}, $P$ is nonempty,
so we can apply the final statement of that corollary
to conclude that $C(S,p)$ is divisible by $p,$ as desired.

What if the $m_i$ are not all distinct?
If a given value $m$ occurs as $h_m$ different $\!m_i\!$'s
(i.e., if after collecting like factors
in~\eqref{d.monom}, $\ad{r}^m(a)$ appears with exponent $h_m),$
then the poset $P$ constructed as above will have $h_m$
\mbox{$\!m{+}1\!$-element} components.
In this situation, for each way the $h_m$ factors $\ad{r}^m(a)$
can arise from an arrow-diagram, we can choose, arbitrarily,
which of those $h_m$ components of $P$
is mapped to which family of $m$ arrows.
This gives $h_m\,!$ possibilities.
We see from this that the coefficient of our monomial~\eqref{d.monom}
in the expansion of $(ar)^{p-1}a$ will now be $C(S,p)/\prod_m h_m!\,.$
But this creates no problem, since each of the $h_m$ is less than $p.$
(Indeed, by~\eqref{d.neq0,p-1},
$p-1>\sum m_i\geq h_m m\geq h_m$ for each $m.)$
So since $C(S,p)$ is divisible by $p,$ $C(S,p)/\prod_m h_m!$
is also, as required.
\end{proof}


\section{Acknowledgements}\label{S.ackn}
I am indebted to Arthur Ogus for asking the question
answered in the above appendix, which
led me into these investigations; to Richard Stanley
for patiently responding to my communications about this
material and pointing me to relevant literature,
and to Bernd Schr\"{o}der
and Stefan Felsner for valuable further correspondence.

\end{document}